\documentclass [12pt]{amsart}

\usepackage[utf8]{inputenc}
\pdfoutput=1
\usepackage{amsmath,amssymb,amsthm,amsfonts}
\usepackage{mathtools}%

\usepackage[english]{babel}%
\usepackage{tabularx}
\usepackage{hyperref}%
\usepackage{bbm}%
\usepackage{bm}%
\usepackage{mathrsfs}%

\usepackage{tikz,graphicx,color}
\usepackage{tikz-cd}%
\usepackage{tikz-3dplot}%
\usetikzlibrary{cd}
\usetikzlibrary{calc}%
\usetikzlibrary{arrows}%
\usetikzlibrary{shapes}%
\usetikzlibrary{patterns}%
\usetikzlibrary{positioning}%
\usetikzlibrary{arrows.meta}
\usetikzlibrary{decorations.markings}
\usetikzlibrary{knots}

\usepackage{epstopdf}%

\usepackage[arrow]{xy}%
\usepackage{diagbox}%
\usepackage{subfig}%
\usepackage{arcs}%
\usepackage{xcolor}%
\usepackage{xspace}
\usepackage{blkarray}

\usepackage[margin=1.0in]{geometry}%

\usepackage{enumitem}
\usepackage{letltxmacro}
\usepackage{thmtools,etoolbox}

\def\myarabic#1{\normalfont(\roman{#1})}
\newlist{theoremlist}{enumerate}{1}
\setlist[theoremlist]{label=\myarabic{theoremlisti},ref={\myarabic{theoremlisti}},itemindent=0pt,labelindent=0pt,
	leftmargin=*,noitemsep}

\makeatletter
\renewcommand{\p@theoremlisti}{\perh@ps{\thetheorem}}
\protected\def\perh@ps#1#2{\textup{#1#2}}
\newcommand{\itemrefperh@ps}[2]{\textup{#2}}
\newcommand{\itemref}[1]{\begingroup\let\perh@ps\itemrefperh@ps\ref{#1}\endgroup}
\makeatother

\usepackage{nameref,hyperref}
\usepackage[capitalize]{cleveref}

\newtheorem{theorem}{Theorem}[section]

\newtheorem{lemma}[theorem]{Lemma}
\newtheorem{proposition}[theorem]{Proposition}
\newtheorem{corollary}[theorem]{Corollary}
\theoremstyle{definition}
\newtheorem{remark}[theorem]{Remark}
\theoremstyle{definition}
\newtheorem{definition}[theorem]{Definition}

\theoremstyle{definition}
\newtheorem{problem}[theorem]{Problem}
\theoremstyle{definition}
\newtheorem{example}[theorem]{Example}

\newcommand{\bp}{\begin{problem}}
	\newcommand{\ep}{\end{problem}}
\newcommand{\bs}{\begin{proof}[Solution]}
	\newcommand{\es}{\end{proof}}
\newcommand{\ra}{\rightarrow}

\newcommand{\Z}{\mathbb{Z}}

\newcommand{\C}{\mathbb{C}}
\newcommand{\R}{\mathbb{R}}
\newcommand{\bv}{\mathrm{b}}
\newcommand{\wv}{\mathrm{w}}

\newcommand{\pl}{\operatorname{Pl}}
\newcommand{\OG}{\operatorname{OG}(n+1,2n)}
\newcommand{\OGdec}{\widetilde{\operatorname{OG}}(n+1,2n)}
\newcommand{\OGpos}{\widetilde{\operatorname{OG}}_{>0}(n+1,2n)}

\newcommand{\IG}{\operatorname{IG}^\Omega(n+1,2n)}
\newcommand{\IGpos}{\operatorname{IG}^\Omega_{>0}(n+1,2n)}
\newcommand{\Grkn}{\operatorname{Gr}(k,n)}
\newcommand{\Grknpos}{\operatorname{Gr}_{>0}(k,n)}

\newcommand{\Grkndec}{\widetilde{\operatorname{Gr}}(k,n)}
\newcommand{\Grkndecpos}{\widetilde{\operatorname{Gr}}_{>0}(k,n)}

\newcommand{\Gr}{\operatorname{Gr}(n+1,2n)}
\newcommand{\Grpos}{\operatorname{Gr}_{>0}(n+1,2n)}

\newcommand{\disk}{\mathbb{D}}

\newcommand{\wt}{\mathrm{wt}}
\newcommand{\Rpos}{\mathbb{R}_{>0}}

\DeclareMathOperator{\Span}{\mathrm{span}}
\DeclareMathOperator{\pf}{\mathrm{pf}}
\DeclareMathOperator{\Spin}{\mathrm{Spin}}
\DeclareMathOperator{\SO}{\mathrm{SO}}
\DeclareMathOperator{\Pin}{\mathrm{Pin}}
\DeclareMathOperator{\Meas}{\mathrm{Meas}}
\newcommand{\Cl}{\mathrm{Cl}}
\newcommand{\cl}{\mathfrak{cl}}
\newcommand{\vect}{x}
\makeatletter
\newcommand{\extp}{\@ifnextchar^\@extp{\@extp^{\,}}}
\def\@extp^#1{\mathop{\bigwedge\nolimits^{\!#1}}}
\makeatother
\newcommand\restr[2]{{
		\left.\kern-\nulldelimiterspace 
		#1 
		\vphantom{\big|} 
		\right|_{#2} 
}}

\crefname{figure}{Figure}{Figures}

\def\figref#1(#2){Figure~\hyperref[#1]{\ref*{#1}(#2)}}

\addtotheorempostheadhook[theorem]{\crefalias{theoremlisti}{theorem}}
\addtotheorempostheadhook[lemma]{\crefalias{theoremlisti}{lemma}}
\addtotheorempostheadhook[proposition]{\crefalias{theoremlisti}{proposition}}
\addtotheorempostheadhook[corollary]{\crefalias{theoremlisti}{corollary}}

\definecolor{calpolypomonagreen}{rgb}{0, 0.6, 0.2}

\newcounter{todofigure}

\tikzset{qvert/.style={draw,black,circle,fill=gray,minimum size=5pt,inner sep=0pt}  } 
\tikzset{bvert/.style={draw,circle,fill=black,minimum size=5pt,inner sep=0pt}  }  
\tikzset{gbvert/.style={draw, gray, circle,fill=gray,minimum size=5pt,inner sep=0pt}  } 
\tikzset{gvert/.style={draw,gray,circle,fill=white,minimum size=5pt,inner sep=0pt}  } 
\tikzset{wvert/.style={draw,circle,fill=white,minimum size=5pt,inner sep=0pt}  } 
\tikzset{fvert/.style={text=MidnightBlue}  } 
\tikzset{sqvert/.style={draw,black,rectangle,fill=black,minimum size=5pt,inner sep=0pt}  } 
\tikzset{lvert/.style={draw,circle,fill=black,minimum size=4pt,inner sep=0pt}  }  
\usetikzlibrary{arrows}
\tikzcdset{arrow style=tikz, diagrams={>={Stealth[round,length=4pt,width=4.95pt,inset=2.75pt]}}}
\usepackage{tikz-cd}
\usetikzlibrary{positioning,calc,arrows,decorations.pathreplacing,decorations.markings,patterns}
\tikzset{mid arrow/.style={postaction={decorate,decoration={
				markings,
				mark=at position .5 with {\arrow{latex}}
	}}},
	mid rarrow/.style={postaction={decorate,decoration={
				markings,
				mark=at position .5 with {\arrow{latex reversed}}
	}}},
}

\tikzset{qvert/.style={draw,black,circle,fill=gray,minimum size=5pt,inner sep=0pt}  } 
\tikzset{bvert/.style={draw,circle,fill=black,minimum size=5pt,inner sep=0pt}  }  
\tikzset{redvert/.style={draw,circle,red,fill=red,minimum size=5pt,inner sep=0pt}  } 
\tikzset{wvert/.style={draw,circle,fill=white,minimum size=5pt,inner sep=0pt}  } \tikzset{bluevert/.style={draw,circle,blue,fill=blue,minimum size=5pt,inner sep=0pt}  }
\tikzset{fvert/.style={text=blue}  }

\definecolor{calpolypomonagreen}{rgb}{0, 0.6, 0.2}

\numberwithin{equation}{section}
\makeatletter
\DeclareRobustCommand{\cev}[1]{%
	\mathpalette\do@cev{#1}%
}
\newcommand{\do@cev}[2]{%
	\fix@cev{#1}{+}%
	\reflectbox{$\m@th#1\vec{\reflectbox{$\fix@cev{#1}{-}\m@th#1#2\fix@cev{#1}{+}$}}$}%
	\fix@cev{#1}{-}%
}
\newcommand{\fix@cev}[2]{%
	\ifx#1\displaystyle
	\mkern#23mu
	\else
	\ifx#1\textstyle
	\mkern#23mu
	\else
	\ifx#1\scriptstyle
	\mkern#22mu
	\else
	\mkern#22mu
	\fi
	\fi
	\fi
}

\makeatother

\begin{document}
	\numberwithin{equation}{section}
	
	\title{{The twist for electrical networks and the inverse problem}}
	\author{Terrence George}
	\address{Department of Mathematics, University of Michigan, Ann Arbor, MI 48103, USA}
	\email{{\href{mailto:georgete@umich.edu}{georgete@umich.edu}}}
	\date{\today}
	
	\begin{abstract}
		We construct an electrical-network version of the twist map for the positive Grassmannian, and use it to solve the inverse problem of recovering conductances from the response matrix.  Each conductance is expressed as a biratio of Pfaffians as in the inverse map of Kenyon and Wilson; however, our Pfaffians are the more canonical $B$ variables instead of their tripod variables, and are coordinates on the positive orthogonal Grassmannian studied by Henriques and Speyer.
	\end{abstract}
	
	\maketitle

	\section{Introduction}

Let $\Gamma = (B \sqcup W,E,F)$ be a {planar} bipartite graph {embedded }in a disk $\disk$ with vertices $\{d_1,\dots,d_n\}$ on the boundary of $\disk$ (and with strand permutation $\pi_{k,n}$; see~\cref{section:bgraph}). Associated with $\Gamma$ is the space $\mathcal X_\Gamma$ of edge weights modulo gauge equivalence. Postnikov \cite{Post} constructed a parameterization of the totally positive Grassmannian $\Grknpos$ using a map $\Meas_\Gamma:\mathcal X_\Gamma \ra \Grknpos$ called \textit{boundary measurement}, where $k:=\#W -\#B$. There is another space $\mathcal A_\Gamma$ of functions $A: F(\Gamma) \ra \Rpos$. Scott \cite{Scott} constructed a function $\Phi_\Gamma:\Grknpos \ra \mathcal A_\Gamma/\Rpos$ assigning to each face of $\Gamma$ a certain Pl\"ucker coordinate. The spaces $\mathcal A_\Gamma$ and $\mathcal X_\Gamma$ are the (positive points of the) $\mathcal A$ and $\mathcal X$ cluster tori of Fock and Goncharov \cite{FockGon}, and there is a canonical map $p_\Gamma: \mathcal A_\Gamma \ra \mathcal X_\Gamma$ that assigns to an edge incident to faces $f,g$ the weight $\frac{1}{A_f A_g}$ (with some modification for boundary edges). Muller and Speyer, generalizing earlier work of {Berenstein, Fomin and Zelevinsky \cite{BFZ}}, and Marsh and Scott \cite{Marsh}, construct automorphisms $\vec \tau$ and $\cev \tau$ of $\Grknpos$, called \textit{right} and \textit{left twists}, that sit in the following commutative diagram {(where $\sim$ denotes homeomorphism)}:
\begin{equation} \nonumber
	\begin{tikzcd} \mathcal A_\Gamma/\Rpos   \arrow[r,"p_\Gamma","\sim"'] & \mathcal X_\Gamma \arrow[d,"\sim"',"\Meas_\Gamma"]\\ \Grknpos \arrow[u,"\sim"',"\Phi_\Gamma"]\arrow[r,bend right=10,"{\vec{\tau}}"',"\sim"]&\Grknpos\arrow[l,bend right=10,"\sim","{\cev{\tau}}"'] 
	\end{tikzcd}.
\end{equation}
{A key application of the twist is a formula for the inverse boundary measurement map; indeed, $\Meas_{\Gamma}^{-1}=p_\Gamma \circ \Phi_\Gamma \circ \cev \tau$. }

{The main goal of this paper is to generalize these results to electrical networks.} Let $G=(V,E)$ be a planar graph embedded in a disk $\disk$ with vertices $\{b_1,\dots,b_n\}$ on the boundary labeled in clockwise cyclic order. A function $c:E(G) \ra \Rpos$ is called a \textit{conductance}, and a pair $(G,c)$ is called an \textit{electrical network}. {Let $\mathcal R_G := \Rpos^{E(G)}$ denote the space of conductances on $G$.} In this paper, we focus on \textit{well connected} electrical networks ({a genericity condition} defined in~\cref{sec:redgr}). 

	The \textit{Laplacian} on $G$ is the linear operator
$
\Delta:\R^{V(G)} \ra \R^{V(G)}
$
defined by 
\[
(\Delta f)(v) := \sum_{e=uv} c(e) (f(v)-f(u))
\]
where the sum is over all edges $uv$ incident to $v$. A function $f: V(G) \ra \R$ is said to be \textit{harmonic} if $(\Delta f)(v)=0$ for all internal vertices $v$ of $G$. Given a function $g: \{b_1,\dots,b_n\} \ra \R$ on the boundary vertices, there is a unique extension of $g$ to a harmonic function $f_g$ on $V(G)$, called the \textit{harmonic extension} of $g$. The linear operator $L : \R^{\{b_1,\dots,b_n\}} \ra \R^{\{b_1,\dots,b_n\}}$ defined by $L(g) = \restr{(-\Delta f_g)}{\{b_1,\dots,b_n\}}$ is called the \textit{response matrix}. It is a negative semidefinite symmetric matrix whose rows and columns sum to $0$. The space of response matrices was characterized by Colin de Verdi\`ere \cite{cdv1} and further studied in \cite{cdv2,cmm,CIM}. The map taking an electrical network to its response matrix is the electrical-network analog of the boundary measurement map. {However, this is more than an analogy and the two constructions are directly related as we now explain.}

The generalized Temperley's bijection of Kenyon, Propp and Wilson \cite{KPW} associates to each electrical network $(G,c)$ a weighted bipartite graph $(G_+,[\wt_+])$, giving an embedding $j_G^+:\mathcal R_G \hookrightarrow \mathcal X_{G_+}$. {The graph $G_+$ has $2n$ boundary vertices and $\#W -\# B =n+1$}. Lam \cite{Lam} studied the composition $\Meas_{G_+} \circ j_G^+ : \mathcal R_G \ra \Grpos$ and showed that the image of $\mathcal R_G$ is a linear slice of $\Grpos$, which was subsequently identified with a positive Lagrangian Grassmannian $\IGpos$ of points in $\Grpos$ that are isotropic for a degenerate skew-symmetric bilinear form $\Omega$ in \cite{BGKT, CGS} (see also \cite{LP1}). \cite[Theorem 1.8]{CGS} explicitly identifies the space of response matrices with $\IGpos$. Therefore, in principle, the inverse problem for electrical networks can be solved using the inverse boundary measurement. However, in practice, the result of inverting the boundary measurement yields a weight on $G_+$ to which one has to apply a complicated gauge transformation to obtain the conductances. The main goal of the paper is to construct a twist map directly for electrical networks without embedding into weighted bipartite graphs.

	Like the space $\mathcal A_\Gamma$, there is a second space $\mathcal B_{G}$ associated with an electrical network parameterized by the $B$ variables. The space $\mathcal B_G$ consists of functions $B:V(G) \sqcup F(G) \ra \Rpos$, and there is a canonical map $q_G: \mathcal B_G \ra \mathcal R_G$ defined as follows. 
Let $e=uv$ be an edge of $G$ and let $f,g$ denote the faces of $G$ incident to $e$. Define $q_G: \mathcal B_G \ra \mathcal R_{G}$ by $c(e):=\frac{B_u B_v}{B_f B_g}$ (cf. Equation (56) in \cite[Section 5.3.1]{GK}). The space $\mathcal B_G$ arises from the study of the cube recurrence, a nonlinear recurrence introduced by Propp \cite{Propp} whose solutions were characterized combinatorially by Carroll and Speyer \cite{CS} (see also \cite{FZ,LP2}). The cube recurrence was further studied by Henriques and Speyer \cite{HenriquesSpeyer}, who related it to the orthogonal Grassmannian $\operatorname{OG}(n+1,2n)$ of $(n+1)$-dimensional subspaces that are coisotropic for a certain symmetric bilinear form $Q$. $\operatorname{OG}(n+1,2n)$ has an embedding in $\C\mathbb P{^{2^{n-1}-1}}\times\C\mathbb P{^{2^{n-1}-1}}$ giving bihomogeneous coordinates on $\operatorname{OG}(n+1,2n)$ called \textit{Cartan coordinates} (which have explicit Pfaffian formulas; see~\cref{sec:pfformulas} ). Henriques and Speyer constructed a homeomorphism $\Psi_G: \OGpos \xrightarrow[]{\sim} \mathcal B_G$ assigning to each vertex and face of $G$ a Cartan coordinate, where $\OGdec$ is the ``affine cone" over $\operatorname{OG}(n+1,2n)$ and $\OGpos$ is the subset where all Cartan coordinates are positive. Our first main result is the following.
\begin{theorem}[cf.~\cref{thm:elecrt}] \label{thm:elecrtintro}
	There is a map $\vec\tau_{\rm elec}$, which we call the electrical right twist, such that the following diagram commutes.
	\[
	\begin{tikzcd} \mathcal B_G   \arrow[r,"q_G"] & \mathcal R_{G} \arrow[d,"\sim"',"\Meas_{G_+} \circ j_{G}^+"]\\ \OGpos \arrow[u,"\Psi_G","\sim"']
		\arrow[r,"\vec\tau_{{\rm elec}}"]
		&\IGpos
	\end{tikzcd}.
	\]
\end{theorem}
The two spaces on the left of the commutative diagram in Theorem~\ref{thm:elecrtintro} have dimension $\binom{n+1}{2}+1$, whereas the two spaces on the right have dimension $\binom n 2$. Therefore, the electrical right twist as defined is not invertible. Our second main result is:

\begin{theorem}[cf.~\cref{thm:elecrtltinv}]\label{thm:intro2}
	There are actions of $\Rpos^{n+1}$ on $\mathcal B_G$ and $\OGpos$ compatible with $\Psi_G$ such that upon taking quotients, $q_G$ and $\vec \tau_{\rm elec}$ are invertible. The inverse $\cev \tau_{\rm elec}$ is called the \textit{electrical left twist} and the following diagram commutes:
	\[
	\begin{tikzcd}[
		column sep={6em}
		] \mathcal B_G/\Rpos^{n+1} \arrow[r,"q_G","\sim"']    & \mathcal R_{G} \arrow[d,"\Meas_{G_+} \circ j_G^+","\sim"']\\ \OGpos/\Rpos^{n+1} \arrow[u,"\Psi_G","\sim"']
		\arrow[r,bend right=10,"\sim","{\vec{\tau}_{\rm elec}}"'] 
		& \IGpos \arrow[l,bend right=10,"\sim","{\cev{\tau}_{\rm elec}}"'] 
	\end{tikzcd}.
	\]
\end{theorem}
As a consequence, we get that the composition $q_G \circ \Psi_G\circ \cev\tau_{\rm elec}$ solves the inverse problem for electrical networks. We work out the inverse map explicitly when $n=3$ in~\cref{sec:recon}.

The inverse problem for electrical networks was first solved using a recursive procedure by Curtis, Ingerman and Morrow \cite{CIM} (see also \cite{net1,net2,net3}). More recently, explicit rational formulas were given by Kenyon and Wilson \cite{KW1,KW}. In the formulas in \cite{KW1,KW}, the conductances are expressed as biratios of certain variables called \textit{tripod variables} which are only defined for special networks called standard networks. The advantage of our construction is that it works for any well connected electrical network and uses the more canonical $B$ variables instead of the tripod variables. We mention that the inverse problem has also been studied in the cylinder \cite{LP3} and the torus \cite{George}.
{On the torus, the inverse map of \cite{George} also factors through $q_G: \mathcal B_G \ra \mathcal R_{G}$ (the $B$ variables are certain Prym theta functions), which further advocates for the naturality of our construction.}

We end the introduction with some open problems. If the graph $G$ is not well connected, then $\mathcal R_G$ parameterizes a smaller electroid cell in $\IG$ which is the intersection of a positroid cell with $\IG$ \cite{Lam}. Muller and Speyer defined the twist map for all postroid cells which suggests the following problem.

\begin{problem}
	Construct a stratified space whose strata are parameterized by $\mathcal B_G$ where $G$ varies over move-equivalence classes of reduced graphs with $n$ vertices on the boundary of the disk. Define an electrical twist map that homeomorphically maps the strata to electroid cells in $\IG$.
\end{problem}

There is another notion of positive orthogonal Grassmannian introduced in \cite{HWX} which was used to parameterize the Ising model by
Galashin and Pylyavskyy \cite{Galpy}. Similarly, there is a positive Lagrangian Grassmannian associated with the cluster side $\mathcal A$ of the Ising model, introduced by Kenyon and Pemantle \cite{KP1,KP2} in relation to the Kashaev recurrence \cite{Kashaev}. The two notions of positive orthogonal/Langrangian Grassmannian do not agree. Instead, we expect the relationship to be as in the table below, where the two spaces in each row are related by twist.

\begin{center}
	\begin{tabularx}{0.8\textwidth} { 
			| >{\centering\arraybackslash}X 
			| >{\centering\arraybackslash}X 
			| >{\centering\arraybackslash}X | }
		\hline
		& cluster $\mathcal A$ side & cluster $\mathcal X$ side \\
		\hline
		dimer models  & positive Grassmannian \cite{Scott}& positive Grassmannian \cite{Post} \\
		\hline
		electrical networks  & positive orthogonal Grassmannian \cite{HenriquesSpeyer}  &  positive Lagrangian Grassmannian \cite{BGKT,CGS} \\
		\hline
		Ising models  & positive Lagrangian Grassmannian \cite{KP1}  & positive orthogonal Grassmannian \cite{Galpy}  \\
		\hline
	\end{tabularx}
\end{center}
\begin{problem}
	Define a twist map for the Ising model relating the positive orthogonal Grassmannian in \cite{Galpy} with the positive Lagrangian Grassmannian in \cite{KP2}. 
\end{problem}
We mention that results relating orthogonal and Lagrangian Grassmannians also appear in \cite{Wang1, Wang2}, but the connection to the above table is unclear.

	\subsection*{Acknowledgments}
	This project originated from conversations with Sunita Chepuri and David Speyer. I also thank David for many discussions on his papers \cite{HenriquesSpeyer} and \cite{MullerSpeyer}. {I also thank the referees for many helpful suggestions.}

	\section{Background on the dimer model and the positive Grassmannian}
	In this section, we review background on the positive Grassmannian, dimer models, and the twist map.

	\subsection{Grassmannians and Pl\"ucker coordinates} \label{sec:grassmann}
	
	The \textit{Grassmannian} $\Grkn$ is the space of $k$-dimensional subspaces of $\C^{n}$. Let $e_1,\dots,e_{n}$ denote the standard basis of $\C^{n}$. For $I = \{i_1 <i_2 < \cdots < i_{k}\} \in \binom{[n]}{k}$, let $e_I:= e_{i_1} \wedge \cdots \wedge e_{i_{k}}$. Then, the $e_I$ form a basis for $\extp^{k} \C^{n}$. The \textit{Pl\"ucker embedding} is the closed embedding $\pl: \Grkn \hookrightarrow \mathbb P(\extp^{k} \C^{n})$ sending a subspace $X$ spanned by $\vect_1,\dots,\vect_{k}$ to $[\vect_1 \wedge \cdots \wedge \vect_{k}]$. The coefficients $\Delta_I(X)$ of $e_I$ in $\vect_1 \wedge \cdots \wedge \vect_{k}$ are called \textit{Pl\"ucker coordinates}. Following \cite{Weng}, we call $\Grkndec:=\{ (X,\vect) \mid X \in \Grkn, \vect \in \extp^k X \}$ the \textit{decorated Grassmannian}. Given $(X,\vect) \in \Grkndec$, we denote the coefficient of $e_I$ in $v$ by $\Delta_I(X,\vect)$. Changing the basis multiplies all the Pl\"ucker coordinates by a common scalar, so they are well-defined functions on $\Grkndec$ but not on $\Grkn$.

	Let $\operatorname{Mat}^\circ(k,n)$ denote the space of $k \times n$ matrices of rank $k$. $\operatorname{GL_k}$ acts on $\operatorname{Mat}^\circ(k,n)$ by left multiplication and we have identifications
	\begin{equation} \label{eq:mattogr}
		\operatorname{GL}_k \backslash \operatorname{Mat}^\circ(k,n) \cong \Grkn ~\text{and}~\operatorname{SL}_k \backslash \operatorname{Mat}^\circ(k,n) \cong \Grkndec
	\end{equation}
	sending the matrix with rows $\vect_1,\dots,\vect_k$ to $\Span(\vect_1,\dots,\vect_k)$ and $(\Span(\vect_1,\dots,\vect_k),\vect_1\wedge \cdots \wedge \vect_k)$ respectively.
	
	Let $\Grkndecpos$ denote the \textit{positive decorated Grassmannian}, the subset of $\Grkndec$ where where all Pl\"ucker coordinates are positive real numbers, and let $\Grknpos$ denote the \textit{positive Grassmannian}, the subset of $\Grkn$ where the ratio of any two Pl\"ucker coordinates is a positive real number.

	\subsection{Planar bipartite graphs in the disk} \label{section:bgraph}
	Let $\Gamma=(B \sqcup W,E,F)$ be a {planar} bipartite graph embedded in a disk $\disk$ with $n$ vertices on the boundary of $\disk$ labeled $d_1,d_2,\dots,d_{n}$ in clockwise cyclic order. {Here, $B$ denotes the set of black vertices, $W$ the set of white vertices, $E$ the set of edges and $F$ the set of faces respectively. Further, we assume that all the boundary vertices are white.} Let $k:=\#W-\#B$. 
	\begin{figure}
	\begin{center}
		\begin{tabular}{cc}
			\includegraphics[width=0.3\textwidth]{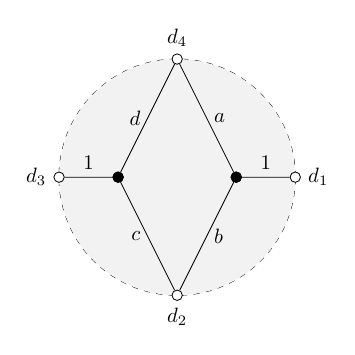}&
			\includegraphics[width=0.3\textwidth]{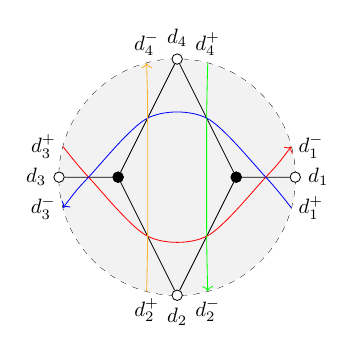}
			\\
			(a)  &(b) 
		\end{tabular}
		\caption{\label{fig:bipgraph} A bipartite graph $\Gamma$ (a) and its medial graph $\Gamma^\times$ with strands (b).}
	\end{center}
\end{figure}
{
The \textit{(oriented) medial graph} $\Gamma^\times$ of $\Gamma$ is the graph obtained as follows. Place $2n$ vertices of $\Gamma^\times$ labeled $d_1^{-},d_1^{+},\dots,d_n^{-},d_n^{+}$ on the boundary of $\disk$ such that $d_i$ is between $d_i^{-}$ and $d_i^+$. Place a vertex $v_e$ in the middle of each edge $e$ of $\Gamma$. Connect $v_e$ and $v_{e'}$ by an edge if they occur consecutively around a face of $\Gamma$. For each $i \in [n]$, connect $d_i^-$ (resp., $d_i^+$) to $v_e$ if $e$ is the last (resp., first) edge in clockwise order incident to $d_i$. By construction, each $d_i^-$ and each $d_i^+$ has degree $1$ and each $v_e$ degree $4$ in $\Gamma^\times$. Orient the edges clockwise around white vertices and counterclockwise around black vertices. Note that this means that edges incident to $d_i^-$ (resp., $d_i^+$) are oriented towards the outside (resp., inside) of $\disk$. 
}

\begin{figure}
	
	\includegraphics[width=0.2\textwidth]{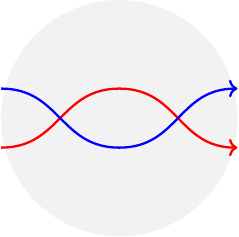}
	
	\caption{\label{fig:parallel_bigon} A parallel bigon. }
\end{figure}
	
{	A \textit{strand} of $\Gamma$ is an oriented walk in $\Gamma^\times$ that either starts and ends at the boundary or is an internal cycle, and at each (degree $4$) vertex of the form $v_e$, the outgoing edge is opposite the incoming one (see Figure~\ref{fig:bipgraph}).
}
 We say that $\Gamma$ is \textit{reduced} (or \textit{minimal}) if:
	\begin{enumerate}
		\item Each strand starts and ends on the boundary, i.e., no strand path is an internal cycle.
		\item No strand has a self-intersection unless it corresponds to a black leaf incident to a boundary white vertex.
		\item Strands do not form ``parallel bigons", i.e., there is no pair of strands that intersect twice in the same direction (Figure~\ref{fig:parallel_bigon}).
	\end{enumerate}
{It is customary to identify a strand with the corresponding oriented walk in $\Gamma$ that uses the edges $e$ of $\Gamma$ in the same order that $v_e$ appear in the strand. Note that such a path turns maximally left at white vertices and maximally right at black vertices, and is called a \textit{zig-zag path}.}

	Let $d_{\pi_\Gamma(i)}^+$ denote the endpoint of the strand that starts at $d_i^-$. Then, $\pi_\Gamma:[n] \ra [n]$ is a permutation called the \textit{strand permutation} of $\Gamma$. Let $\pi_{k,n}:[n] \ra [n]$ be the permutation $(k+1,k+2,\dots,n,1,2,\dots,k-1)$.
	
	\begin{remark}
		If $\pi_\Gamma(i)=i$, then we also have to specify a color for $i$, but this does not occur in $\pi_{k,n}$. 
	\end{remark}
	\begin{figure}
		\begin{center}
			\begin{tabular}{ccc}
				\includegraphics[width=0.4\textwidth]{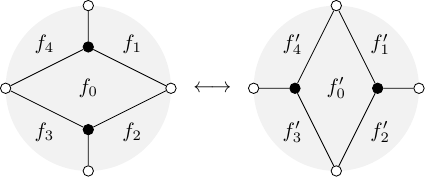}
				& \qquad &
				\includegraphics[width=0.4\textwidth]{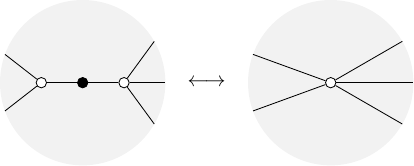}
				\\
				(a) The spider move. & & (b) The contraction-uncontraction move. 
			\end{tabular}
		\end{center}
		\caption{\label{fig:bipartitemoves} Moves for bipartite graphs. }
	\end{figure}

	We say that two planar bipartite graphs $\Gamma$ and $\Gamma'$ are \textit{move-equivalent} if they are related by the moves shown in Figure~\ref{fig:bipartitemoves}. Each move $\Gamma \rightsquigarrow \Gamma'$ induces a canonical bijection between $F(\Gamma)$ and $F(\Gamma')$; we denote the face of $\Gamma'$ corresponding to the face $f$ of $\Gamma$ by $f'$. Postnikov \cite{Post} and Thurston \cite{Thurston} showed that two reduced bipartite graphs are move-equivalent if and only if they have the same strand permutation.

	\subsection{Dimer models and boundary measurement} \label{sec:dimerboundary}
	
	Let $\wt:E(\Gamma) \ra \Rpos$ be a function called an \textit{edge weight}. Two edge weights ${\rm wt}_1$ and ${\rm wt}_2$ are said to be \emph{gauge equivalent} if there is a function $g:B(\Gamma) \sqcup W(\Gamma) \ra \Rpos$ that is equal to $1$ on the boundary vertices such that for every edge $e= {\rm b}{\rm w}$ with ${\rm b} \in B(\Gamma), {\rm w}\in W(\Gamma)$, we have ${\rm wt}_2(e)=g({\rm b})^{-1}{\rm wt}_1(e) g({\rm w}) $. Let $\mathcal X_\Gamma := \Rpos^{E(\Gamma)}/\text{gauge}$ denote the space of edge weights on $\Gamma$ modulo gauge equivalence. We denote the gauge equivalence class of $\wt$ by $[\wt]$. A pair $(\Gamma,[\wt])$ with $[\wt] \in \mathcal X_\Gamma$ is called a \textit{dimer model}.
	
	For a face $f$ of $\Gamma$ with {counterclockwise-oriented} boundary ${\rm w}_1 \xrightarrow[]{e_1} {\rm b}_1 \xrightarrow[]{e_2} {\rm w}_2 \xrightarrow[]{e_3} {\rm b}_2 \xrightarrow[]{e_4} \cdots \xrightarrow[]{e_{2k-2}} {\rm w}_k \xrightarrow[]{e_{2k-1}} {\rm b}_k \xrightarrow[]{e_{2k}} {\rm w}_1$, let
	\[
	X_f:=\prod_{i=1}^k \frac{\wt(e_{2i})}{\wt(e_{2i-1})} 
	\]
	denote the alternating product of the edge weights around the boundary of $f$. The $X_f$'s are invariant under gauge equivalence and provide coordinates on $\mathcal X_\Gamma$ satisfying the relation $\prod_{f \in F(\Gamma)} X_f=1$, so $\mathcal X_\Gamma \cong \Rpos^{\#F(\Gamma)-1}$.
	
	A move $\Gamma \rightsquigarrow \Gamma'$ induces a homeomorphism $\mathcal X_{\Gamma} \xrightarrow[]{\sim } \mathcal X_{\Gamma'}$ defined as follows:
	\begin{enumerate}
		\item Spider move at a face $f_0$: The homeomorphism $\mathcal X_{\Gamma} \xrightarrow[]{\sim } \mathcal X_{\Gamma'}$ is given by
		\[
		X_{f_0'} := \frac{1}{X_{f_0}}, X_{f_1'}:= X_{f_1}(1+X_{f_0}),X_{f_2'}:= \frac{X_{f_2}}{(1+\frac{1}{X_{f_0}})},X_{f_3'}:= X_{f_3'}(1+X_{f_0}),X_{f_4}:= \frac{X_{f_4'}}{(1+\frac{1}{X_{f_0}})},
		\]
		and $X_{f'}:=X_f$ for $f' \in F(\Gamma') \setminus \{f_0',f_1',f_2',f_3',f_4'\}$.
		\item Contraction-uncontraction move: The homeomorphism $\mathcal X_{\Gamma} \xrightarrow[]{\sim } \mathcal X_{\Gamma'}$ is $X_{f'}:= X_f$ for all $f' \in F(\Gamma')$.
	\end{enumerate}

	Given a strand permutation $\pi$, let $\mathcal X_\pi:= \bigsqcup_{\pi_\Gamma=\pi} \mathcal X_{\Gamma}\Big/\text{moves}$ denote the \textit{space of dimer models}, where the union is over all reduced bipartite graphs $\Gamma$ with strand permutation $\pi$.
	
	A \textit{dimer cover} (or \textit{almost perfect matching}) of $\Gamma$ is a subset of $E(\Gamma)$ that uses each internal vertex of $\Gamma$ and a subset of the boundary vertices exactly once. The \textit{weight} $\wt(M)$ of a dimer cover $M$ is defined to be $\prod_{e \in M} \wt(e)$. For a dimer cover $M$, let 
	\[
	\partial M  := \{ i\in [n] \mid d_i~\text{is not used by $M$}\} \in \binom{[n]}{k},
	\]
{where $n$ and $k$ are as in Section~\ref{section:bgraph}.} For $I \in \binom{[n]}{k}$, define the \textit{dimer partition function}
	\[
	Z_I:=\sum_{M \mid \partial M=I} \wt(M).
	\]
	Postnikov \cite{Post} defined the \textit{boundary measurement map}
	\[
	\Meas_\Gamma: \mathcal X_\Gamma \ra \mathbb P(\extp^{k} \C^{n})
	\]
	sending $[\wt]$ to $[\sum_{I \in \binom{[n]}{k}} Z_I e_I]$. $\Meas_\Gamma$ is well-defined, since the gauge equivalence multiplies all $Z_I$'s by a scalar. The following theorem is due to Postnikov \cite{Post} in a different language (see also \cite{PSW} and \cite[Corollary 7.14]{Lam2}).

	\begin{theorem}
		For a reduced $\Gamma$ with $\pi_\Gamma = \pi_{k,n}$, $\Meas_{\Gamma}: \mathcal X_{\Gamma} \xrightarrow[]{\sim} \Grknpos$ is a homeomorphism. If $\Gamma$ and $\Gamma'$ are related by a move, then the following diagram commutes:
		\[
		\begin{tikzcd}
			\mathcal X_{\Gamma}\arrow[dr,"\Meas_{\Gamma}","\sim"'] \arrow[dd,"\text{move}"',"\sim"]&\\ &\Grknpos\\ \mathcal X_{\Gamma'} \arrow[ur,"\Meas_{\Gamma'}"',"\sim"]& 
		\end{tikzcd}.
		\]
		Therefore, the maps $\Meas_\Gamma$ glue to a homeomorphism $\Meas: \mathcal X_{\pi_{k,n}} \xrightarrow{\sim} \Grknpos$.
	\end{theorem}

	\begin{example} \label{example:bip}
		Let $(\Gamma,\wt)$ be the weighted bipartite graph shown in Figure \ref{fig:bipgraph}(a). From the strands shown in Figure \ref{fig:bipgraph}(b), obtain the strand matching to be $\pi_{2,4}$. The boundary measurement map sends $[\wt]$ to $\left[	a e_{12}+ (ac+bd) e_{13} + be_{14} + d e_{23} + e_{24} +c e_{34}\right]$, which is the image under $\pl$ of
		\begin{equation} \label{eq:matrixX}
			X:=	\text{row span}	\begin{bmatrix}
				b&1&c&0\\
				-a&0&d&1
			\end{bmatrix}.
		\end{equation}
	\end{example}

	\subsection{\texorpdfstring{$A$}{A} variables} \label{sec:Avariables}
	
	Let $\mathcal A_{\Gamma}:= \Rpos^{F(\Gamma)}$ denote the space of functions $A:F(\Gamma) \ra \Rpos$. A move $\Gamma \rightsquigarrow \Gamma'$ induces a homeomorphism $\mathcal A_{\Gamma} \xrightarrow[]{\sim } \mathcal A_{\Gamma'}$ as follows:
	\begin{enumerate}
		\item Spider move at a face $f_0$: The homeomorphism $\mathcal A_{\Gamma} \xrightarrow[]{\sim } \mathcal A_{\Gamma'}$ is given by the cluster mutation formula
		\[
		A_{f_0'} := \frac{A_{f_1} A_{f_3}+A_{f_2}A_{f_4}}{A_{f_0}} 
		\]
		and $A_{f'}:=A_f$ for $f' \in F(\Gamma') \setminus \{f_0'\}$.
		\item Contraction-uncontraction move: The homeomorphism $\mathcal A_{\Gamma} \xrightarrow[]{\sim } \mathcal A_{\Gamma'}$ is $A_{f'}:= A_f$ for all $f' \in F(\Gamma')$.
	\end{enumerate}

	Let $\mathcal A_\pi := \bigsqcup_{\pi_\Gamma=\pi} \mathcal A_{\Gamma}\Big/\text{moves}$.
	
	\begin{remark}
		The spaces $\mathcal X_\Gamma$ and $\mathcal A_\Gamma$ are the positive points of the $\mathcal X$ and $\mathcal A$ {cluster tori} associated with $\Gamma$ respectively (see \cite{FockGon}), and $\mathcal X_\pi$ and $\mathcal A_\pi$ are the positive points of the $\mathcal X$ and $\mathcal A$ cluster varieties respectively. Since the cluster varieties do not appear directly in this paper, we have chosen to denote the positive points by $\mathcal X_\Gamma$ instead of $\mathcal X_\Gamma(\Rpos)$ etc.
	\end{remark}	
	\begin{figure}
		\begin{center}
			\begin{tabular}{ccc}
				\includegraphics[width=0.3\textwidth]{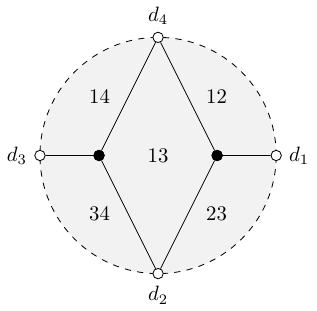} &\includegraphics[width=0.3\textwidth]{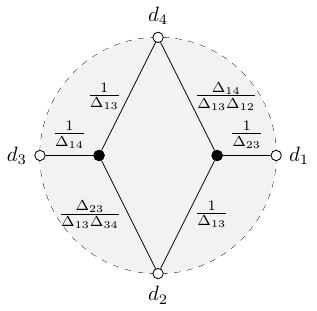}
				&\includegraphics[width=0.3\textwidth]{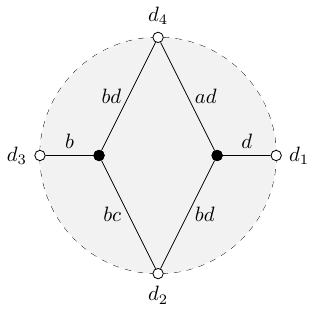}\\
				(a) Face labels. &(b) $\Phi_\Gamma \circ p_\Gamma$. &(c) $p_\Gamma \circ \Phi_\Gamma \circ \cev \tau$.
			\end{tabular}
			\caption{\label{fig:facelabels} Inverting the boundary measurement map for the graph in Figure \ref{fig:bipgraph}.}
		\end{center}
	\end{figure}	
	\begin{definition}
			{The faces of the medial graph $\Gamma^\times$ are in bijection with $B(\Gamma)\sqcup W(\Gamma) \sqcup F(\Gamma)$. We say that $f \in F(\Gamma)$ is to the left of a strand if the corresponding face of $\Gamma^\times$ is to the left of the strand.} For each face $f$ of $\Gamma$, define the \textit{(target) face label} 
		\[
		S(f) := \{i \in [n] \mid f~\text{is on the left of the strand ending at }d_i^+\}.
		\]
	
	\end{definition}
	For each face $f$, $S(f)$ is a $k$-element subset of $[n]$. Let $f_1^-,\dots,f_{n}^-$ denote the boundary faces of $\Gamma$ so that $f_i^-$ is between $d_{i-1}$ and $d_{i}$. If $\pi_\Gamma=\pi_{k,n}$, then $S(f_{i}^-) = \{i,i+1,\dots,i+k-1\}$ are the cyclically consecutive subsets. 
	
	\begin{example}
		For the graph in Figure \ref{fig:bipgraph}(a), using the strands shown in Figure \ref{fig:bipgraph}(b), we compute the face labels as shown in Figure \ref{fig:facelabels}. 
	\end{example}
	
	Scott \cite{Scott} defined the map
	\[
	\Phi_\Gamma: \Grkndecpos \ra \mathcal A_{\Gamma} 
	\]
	sending $(X,\vect)$ to $(\Delta_{S(f)}(X,\vect))_{f \in F(\Gamma)}$. 
	\begin{theorem}[Scott, {\cite[Theorem 4]{Scott}}] \label{thm:scott}
		For every reduced $\Gamma$ with $\pi_\Gamma = \pi_{k,n}$, $\Phi_\Gamma:\Grkndecpos \xrightarrow{\sim} \mathcal A_{\Gamma}$ is a homeomorphism. 	If $\Gamma_1$ and $\Gamma_2$ are related by a move, then
		\begin{equation*} 
			\begin{tikzcd}
				&	\mathcal A_{\Gamma_1}\arrow[dr] \arrow[dd,"\text{move}","\sim"']&\\ \Grkndecpos\arrow[dr,"\Phi_{\Gamma_2}"',"\sim"]\arrow[ur,"\Phi_{\Gamma_1}","\sim"']&&\mathcal A_{\pi_{k,n}}\\& \mathcal A_{\Gamma_2} \arrow[ur] & 
			\end{tikzcd}.
		\end{equation*}
		commutes, so we obtain a well-defined homeomorphism $\Phi: \Grkndecpos \xrightarrow{\sim}  \mathcal A_{\pi_{k,n}}$.
	\end{theorem}

	\subsection{Twist} \label{sec:twist}
	We introduce the twist map defined by Marsh and Scott \cite{Marsh} and generalized by Muller and Speyer \cite{MullerSpeyer}. We follow the normalization conventions of \cite{MullerSpeyer}. Let $M$ be a $k\times n$ matrix whose $k \times k$ minors are all nonzero. For any $i \in [n]$, let $M_i$ denote the $i$th column of $M$. We extend this definition to all $i \in \Z$ by defining $M_{i} := M_{\overline i}$ where $\overline i \in [n]$ is the reduction of $i \in \Z$ modulo $n$. Let $\langle \cdot, \cdot \rangle$ denote the standard inner product on $\R^{k}$. 
	\begin{definition}\label{def:twist}
		The \textit{right twist} of $M$ is the $k \times n$ matrix $\vec{\tau}(M)$ whose column $\vec{\tau}(M)_i$ is defined by
		\[
		\langle \vec{\tau}(M)_i, M_i \rangle = 1 ~\text{and}~\langle \vec{\tau}(M)_i, M_j \rangle=0~\text{for}~ i <j \leq i+k-1.
		\]
		Similarly, the \textit{left twist} of $X$ is the $k \times n$ matrix $\cev{\tau}(M)$ whose column $\cev{\tau}(M)_i$ is defined by
		\[
		\langle \cev{\tau}(M)_i, M_i \rangle = 1 ~\text{and}~\langle \cev{\tau}(M)_i, M_j \rangle=0~\text{for}~ i-k+1  \leq j <i.
		\] 
	\end{definition}

	\begin{theorem} [Muller and Speyer, {\cite[Corollary 6.8]{MullerSpeyer}}] \label{thm:ms1}
		Under the identifications (\ref{eq:mattogr}), the right and left twists descend to mutually inverse homeomorphisms of $\widetilde{\operatorname{Gr}}_{>0}(k,n)$ and $\Grknpos$.
	\end{theorem}
	
	\begin{definition}
		We denote the right twist of $(X,\vect) \in \Grkndecpos$ (resp., $X \in \Grknpos$) by $\vec \tau(X,\vect)$ (resp., $\vec \tau (X)$), and similarly for the left twist.
	\end{definition}
	\begin{example} \label{example:twist}
		The left twist of $X$ in (\ref{eq:matrixX}) is $
		\cev \tau(X) = \text{row span}\begin{bmatrix}
			\frac 1 b & 1&0&-\frac d c\\0& \frac b a & \frac 1 d&1
		\end{bmatrix}.
		$
	\end{example}
	
	\begin{definition}\label{def:atox}
		Let $\Gamma$ be a reduced bipartite graph with $\pi_\Gamma = \pi_{k,n}$. Define the map $p_\Gamma: \mathcal A_\Gamma \ra \mathcal X_\Gamma$ sending $A$ to $[\wt]$ as follows. Let $e \in E(\Gamma)$ be an edge and let $f,g \in F(\Gamma)$ be the two faces incident to $e$. Define
		\[
		\wt (e) := \begin{cases} \frac{1}{A_f A_g} &\text{if $e$ is not incident to a boundary white vertex},\\
			\frac{A_{f_i^-}}{A_f A_g}&\text{if $e$ is incident to boundary white vertex $d_i$},
		\end{cases}
		\]
		where $f_i^-$ is the boundary face of $\Gamma$ between $d_{i-1}$ and $d_{i}$.
	\end{definition}	
	\begin{theorem} \cite[Theorem 7.1 and Remark 7.2]{MullerSpeyer} \label{thm:MSmain}
		Let $\Gamma$ be reduced bipartite graph with $\pi_\Gamma=\pi_{k,n}$. The following diagrams commute.
		\begin{equation} \label{twistp}
			\begin{tikzcd} \mathcal A_\Gamma   \arrow[r,"p_\Gamma"] & \mathcal X_\Gamma \arrow[d,"\sim"',"\Meas_\Gamma"]\\ \widetilde{\operatorname{Gr}}_{>0}(k,n) \arrow[swap]{r}{\vec{\tau}} \arrow[u,"\sim"',"\Phi_\Gamma"]&\Grknpos 
			\end{tikzcd},
			\quad 
			\begin{tikzcd} \mathcal A_\Gamma/\Rpos   \arrow[r,"p_\Gamma","\sim"'] & \mathcal X_\Gamma \arrow[d,"\sim"',"\Meas_\Gamma"]\\ \Grknpos \arrow[u,"\sim"',"\Phi_\Gamma"]\arrow[r,bend right=10,"{\vec{\tau}}"',"\sim"]&\Grknpos\arrow[l,bend right=10,"\sim","{\cev{\tau}}"'] 
			\end{tikzcd}.
		\end{equation}
		In the diagram on the {right}, the quotient is by the action of $\Rpos$ on $\mathcal A_\Gamma$ multiplying all the $A$ variables by a scalar.
	\end{theorem}

	\begin{remark}
		The map $p_\Gamma$ is an incarnation of the canonical map between $\mathcal A$ and $\mathcal X$ cluster varieties in Fock and Goncharov \cite{FockGon}.
	\end{remark}	
	
	\begin{example}
		Recall Examples~\ref{example:bip} and \ref{example:twist}. The Pl\"ucker coordinates of $\cev \tau (X)$ are 
		\[
		\Delta_{12} = \frac 1 a, \Delta_{13}=\frac{1}{bd}, \Delta_{14}=\frac 1 b, \Delta_{23}= \frac 1 d, \Delta_{24}= 1+\frac{bd}{ac}, \Delta_{34}= \frac 1 c.
		\]
		The compositions $p_\Gamma \circ \Phi_\Gamma$ and $p_\Gamma \circ \Phi_\Gamma \circ \cev \tau$ are shown in Figure~\ref{fig:facelabels}(b) and Figure~\ref{fig:facelabels}(c) respectively. The weights in Figure~\ref{fig:bipgraph}(a) and Figure~\ref{fig:facelabels}(c) are easily seen to be gauge equivalent.
	\end{example}
	
	\begin{definition} \label{def:action}
		For $t = (t_1,\dots,t_n)  \in \Rpos^n$ and $X \in \Grpos$, let $t \cdot X \in \Grpos$ denote the point obtained as follows. Let $M$ be a $k \times n$ matrix such that $X$ is the row span of $M$. Then, $t \cdot X$ is the row span of the matrix $t \cdot M$ defined by $(t \cdot M)_i:=t_i M_i$. 	
	\end{definition}

	Let $\Gamma$ be a reduced bipartite graph with $\pi_\Gamma = \pi_{k,n}$ and let $[\wt] \in \mathcal X_\Gamma$. Let $\Rpos^n$ act on $\mathcal X_\Gamma$ by multiplying the weights of all edges incident to $d_i$ by $\frac{1}{t_i}$. The following lemma is used in the proof of~\cref{thm:elecrt}. 
	\begin{lemma} \label{lem:tequiv}
		The map $\Meas_\Gamma:\mathcal X_\Gamma \ra \Grknpos$ is $\Rpos^{n}$ equivariant.
	\end{lemma}
	\begin{proof}
		We have $\Delta_I(t \cdot X) = (\prod_{i \in I}t_i) \Delta_I(X)$. On the other hand, \begin{align*}
			\Meas_\Gamma(t \cdot [\wt])&=\left[\sum_{I \in \binom{[n]}{k}} \left(\prod_{i \notin I \mid d_i \in W(\Gamma)}t_i \right)Z_I e_I\right]\\
			&=\left[\sum_{I \in \binom{[n]}{k}} \left(\prod_{i \in I }t_i \right)Z_I e_I\right],
		\end{align*}
		where in the second equality we rescaled by $\prod_{i \in [n] \mid d_i \in W(\Gamma)}t_i$.
	\end{proof}

	The following two properties of the twist will be required later.
	\begin{proposition}[Muller and Speyer, {\cite[(9) in the proof of Proposition 6.6 and Proposition 6.1]{MullerSpeyer}}] \label{prop:msinv}
		Let $X \in \Grknpos$. 
		\begin{enumerate}
			\item For any boundary face $f_i^-$, we have
			$
			\Delta_{S(f_i^-)}(\vec \tau(X))=\frac{1 }{	\Delta_{S(f_i^-)}(X)}.
			$
			\item If $t=(t_1,\dots,t_n) \in \Rpos^n$, then $\vec \tau(t \cdot X) = t^{-1} \cdot \vec\tau(X)$, where $t^{-1}:=(\frac{1}{t_1},\dots,\frac{1}{t_n})$.
		\end{enumerate}
		
	\end{proposition}

	\section{Electrical networks}	
	
	\begin{figure}[ht]
		\begin{tabular}{ccc}
			\includegraphics[width=0.3\textwidth]{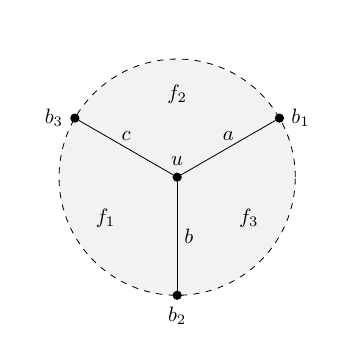} &\includegraphics[width=0.3\textwidth]{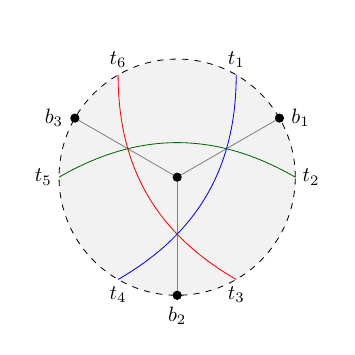}
			&\includegraphics[width=0.3\textwidth]{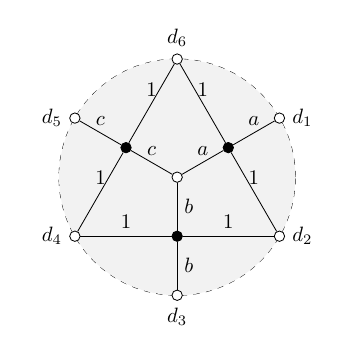}\\
			(a) An electrical network $(G,c)$. &(b) $G^\times$. &(c) $(G_+,[\wt_+])$.
		\end{tabular}
		\caption{\label{fig:facenetwork3} An electrical network with $n=3$ and its associated graphs. The three strands of $G$ are given different colors.}
	\end{figure}	
	
	\subsection{Reduced graphs in the disk}
	\label{sec:redgr}	
	Let $G=(V,E,F)$ be a planar graph embedded in the disk $\disk$ with $n$ vertices on the boundary labeled $b_1,b_2,\dots,b_n$. The \textit{medial graph} $G^\times$ of $G$ is the graph obtained as follows. Place $2n$ vertices of $G^\times$ labeled $t_1,t_2,\dots,t_{2n}$ on the boundary of $\disk$ such that $b_i$ is between $t_{2i-1}$ and $t_{2i}$ and a vertex $v_e$ in the middle of each edge $e$ of $G$. Connect $v_e$ and $v_{e'}$ by an edge if they occur consecutively around a face of $G$. For each $i \in [n]$, connect $t_{2i-1}$ (resp., $t_{2i}$) to $v_e$ if $e$ is the last (resp., first) edge in clockwise order incident to $b_i$. By construction, each $t_i$ has degree $1$ and each $v_e$ degree $4$ in $G^\times$. A \textit{strand} of $G$ is a {maximal walk in $G^\times$ that goes ``straight through" every vertex $v_e$ in it, i.e., if $e_1^\times,e_2^\times$ are two consecutive edges of $G^\times$ in the walk with common vertex $v_e$, then $e_1^\times$ and $e_2^\times$ are opposite each other with respect to the cyclic order of edges around $v_e$ (which makes sense since $v_e$ has degree $4$). Unlike strands in a bipartite graphs, strands in $G$ are unoriented.}
	
	\begin{example}
		Figure~\ref{fig:facenetwork3}(b) shows the medial graph of the electrical network in Figure~\ref{fig:facenetwork3}(a).
	\end{example}

	The graph $G$ is called \textit{reduced} if:
	\begin{enumerate}
		\item Every strand starts and ends at a boundary vertex, i.e., no strand is an internal cycle.
		\item Strands have no self-intersections.
		\item There is no pair of strands that intersect twice.
	\end{enumerate}

	\begin{figure}
		
		\includegraphics[width=0.6\textwidth]{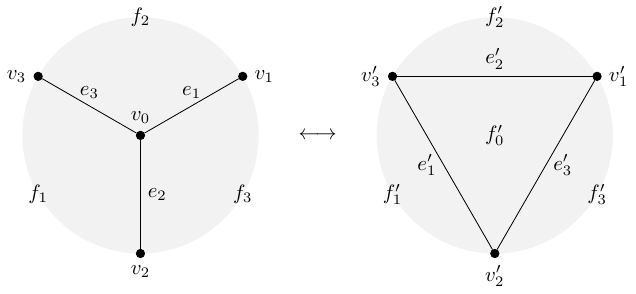}
		
		\caption{\label{fig:ydmove} The Y-$\Delta$ move. }
	\end{figure}
	
	The \textit{medial pairing} of $G$ is the matching on $[2n]$ defined by
	\[
	\tau_G:=\{\{i,j\} \mid \text{there is a strand between $t_i$ and $t_j$} \}.
	\]
	In this paper, we only consider reduced graphs $G$ with medial pairing $\tau_n:=\{\{1,n+1\},\{2,n+2\},\dots,\{n,2n\}\}$; such graphs are called \textit{well connected}. {If $G$ is well connected, since the edges of $G$ are in bijection with crossings of strands and any two strands cross exactly once, $G$ has $\binom n 2$ edges.}
	
\begin{example}
	For the electrical network in Figure~\ref{fig:facenetwork3}(a), the medial graph is shown in Figure~\ref{fig:facenetwork3}(b), from which we see that $G$ is reduced with medial pairing $\tau_3$.
\end{example}
	
	We say that $G$ and $G'$ are \textit{move-equivalent} if they are related by a sequence of Y-$\Delta$ moves (Figure~\ref{fig:ydmove}). A Y-$\Delta$ move $G \rightsquigarrow G'$ induces canonical bijections $V(G) \sqcup F(G) \xrightarrow[]{\sim} V(G') \sqcup F(G')$ and $E(G) \xrightarrow[]{\sim} E(G')$. Two graphs $G$ and $G'$ are move-equivalent if and only if they have the same medial pairing \cite{cdv1}.

	\subsection{The space of electrical networks and the positive Lagrangian Grassmannian}
	\label{sec:gtogplus}
	Let $c:E(G) \ra \Rpos$ be a function called \textit{conductance}, and let $\mathcal R_{G} := \Rpos^{E(G)} {\cong \Rpos^{\binom n 2}}$ be the space of conductances on $G$. A pair $(G,c)$ with $c \in \mathcal R_{G}$ is called an \textit{electrical network}.
	
	A Y-$\Delta$ move $G \rightsquigarrow G'$ induces a homeomorphism $\mathcal R_{G} \xrightarrow[]{\sim} \mathcal R_{G'}$ given by
	\[
	c(e_1') := \frac{c(e_2) c(e_3)}{C},c(e_2') := \frac{c(e_1) c(e_3)}{C},c(e_3') := \frac{c(e_1) c(e_2)}{C},
	\]
	where $C:=c(e_1) c(e_2) + c(e_1) c(e_3) + c(e_2) c(e_3)$ and the edges are labeled as in Figure~\ref{fig:ydmove}, {while the conductances of edges not involved in the Y-$\Delta$ move are unchanged.} Let $\mathcal R_n := \bigsqcup_{\tau_G=\tau_n} \mathcal R_{G}\Big/\text{moves}$ denote the \textit{space of electrical networks}.
	
		\begin{figure}
		\begin{tabular}{ccc}
			\includegraphics[width=0.6\textwidth]{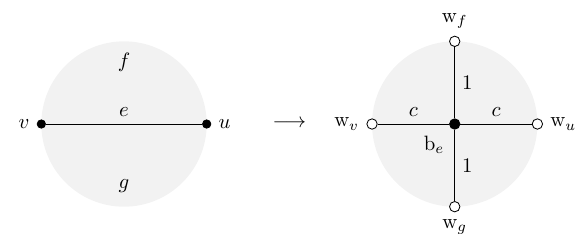} 
		\end{tabular}
		\caption{\label{fig:temperley1} The procedure to obtain $G_+$ from $G$, where $c$ is the conductance of the edge $e$.}
	\end{figure}
	
	The generalized Temperley's bijection of \cite{KPW} associates a dimer model $(G_+,[\wt_+])$ to $(G,c)$ as follows. Place a black vertex $\bv_e$ in the middle of every edge $e$ of $G$, a white vertex $\wv_v$ at every vertex $v$ of $G$, {a white vertex $\wv_f$ in the middle of every internal face $f$, and a white vertex $\wv_f$ in the middle of the intersection of the boundary of $\disk$ with $f$ for every boundary face $f$ of $G$.} If $v$ is a vertex of $G$ incident to edge $e$, draw an edge $\bv_e \wv_v$ and assign $\wt_+(\bv_e \wv_v):=c(e)$. If $f$ is a face of $G$ incident to $e$, draw an edge $\bv_e \wv_f$ and assign $\wt_+(\bv_e \wv_f):=1$ (see Figure~\ref{fig:temperley1}). {$G_+$ has $2n$ boundary white vertices which we label $d_1,\dots,d_{2n}$ in clockwise cylic order as follows:
	\[
	d_{2i-1}:=\wv_{b_i} ~\text{and}~d_{2i}:=\wv_{f_i},
	\]
	where $f_i$ denotes the boundary face between $b_i$ and $b_{i+1}$.}
			\begin{figure}
		\begin{tabular}{ccc}
			\includegraphics[width=0.6\textwidth]{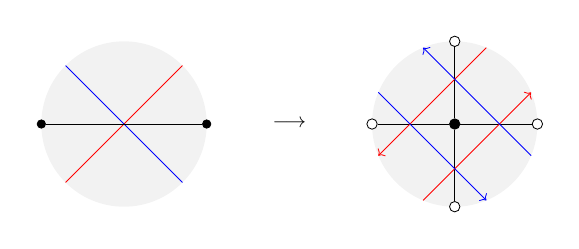} 
		\end{tabular}
		\caption{\label{fig:temperley2} The $1:2$ correspondence between strands in $G$ and strands in $G_+$.}
	\end{figure}
	
	{An Euler characteristic computation shows that $\#W(G_+)-\#B(G_+)=n+1$. Moreover, there is a $1:2$ correspondence between strands in $G$ and strands in $G_+$ (Figure~\ref{fig:temperley2}) which can be seen locally. For $i \in [n]$, let $\beta_i$ denote the strand in $G$ between $t_{i}$ and $t_{n+i}$. Let $\alpha_i$ denote the strand of $G_+$ ending at $d_i^+$. From the local picture in Figure~\ref{fig:temperley2}, we see that $\alpha_i$ starts at $d_{n+i-1}^-$, and that the two strands in $G_+$ that correspond to $\beta_i$ are $\alpha_i$ and $\alpha_{n+i}$.

}
	
	\begin{example}
		The weighted bipartite graph associated to the electrical network in Figure~\ref{fig:facenetwork3}(a) is shown in Figure~\ref{fig:facenetwork3}(c).
	\end{example}	
	
	\begin{remark}
		The notation $G_+$ is inspired by the notation $G_\square$ for the Ising graph in \cite{Galpy}, since we are replacing each edge of $G$ with a $+$.
	\end{remark}

	The map $(G,c) \mapsto (G_+,[\wt_+])$ defines an inclusion $j_G^+: \mathcal R_{G} \hookrightarrow \mathcal X_{G_+}$. 
	
	\begin{proposition} [Goncharov and Kenyon, {\cite[Lemma 5.11]{GK}}]
		If $G$ and $G'$ are related by a Y-$\Delta$ move, then there is a sequence of {moves for bipartite graphs} relating $G_+$ and $G_+'$ making the following diagram commute.
		\[
		\begin{tikzcd} \mathcal R_{G} \arrow[r,hook,"j_{G}^+"]\arrow[d,"\text{Y-$\Delta$ move}"',"\sim"] & \mathcal X_{G_+}\arrow[d,"\text{moves}","\sim"']\\ \mathcal R_{G'} \arrow[r,hook,"j_{G'}^+"] &\mathcal X_{G_+'}
		\end{tikzcd}.
		\]
		Therefore, the inclusions $j_G^+$ glue to an inclusion $j^+:\mathcal R_n \ra \mathcal X_{\pi_{n+1,2n}}$.
	\end{proposition}

	Let $\Omega:\R^{2n} \times \R^{2n} \ra \R$ be the degenerate skew symmetric bilinear form 
	\begin{multline} \Omega(x,y) = 		\sum_{i=1}^n (x_{2i-1} y_{2i} - x_{2i} y_{2i-1}) + \sum_{i=1}^{n-1} (x_{2i+1} y_{{2i}} - x_{{2i}} y_{2i+1}) + (-1)^n (x_1 y_{{2n}} - x_{2n } y_1) .  \label{OmegaFormula} \end{multline}
	We say that $X \in \Gr$ is \textit{isotropic} for $\Omega$ if $\Omega(x,y)=0$ for any $x,y \in X$. Let $\IG$ be the \textit{Lagrangian Grassmannian} of isotropic subspaces inside $\Gr$ and $\IGpos := \IG \cap \Grpos$ the \textit{positive Lagrangian Grassmannian}.
	
	{
	\begin{remark}
		The form $\Omega$ has a two-dimensional kernel which must be contained in every isotropic $(n+1)$-dimensional subspace. Therefore, if we quotient by the kernel, we get that $\IG$ is isomorphic to the Lagrangian Grassmannian $\operatorname{LG}(n-1,2n-2)$. The total positivity structure of $\IGpos$ is also non-standard (see~\cite[Section 5]{CGS} for further discussion).
	\end{remark}
}
	
	 The following result was independently proved by Bychkov, Gorbounov, Kazakov and Talalaev \cite{BGKT} and Chepuri, George and Speyer \cite{CGS}, following earlier results of Lam \cite{Lam}.
	
	\begin{theorem}\label{lamthmmain}
		The composition $\Meas_{G_+} \circ j_G^+:\mathcal R_{G}\xrightarrow[]{\sim}\IGpos$ is a homeomorphism. 
	\end{theorem}

Therefore, we have a commuting diagram
	\begin{equation*} 
		\begin{tikzcd} \mathcal R_G  \arrow[r,hook,"j_{G}^+"]\arrow[d,"\sim","\Meas_{G_+} \circ j_G^+"'] & \mathcal X_{G_+} \arrow[d,"\sim"',"\Meas_{G_+}"]\\ \IGpos \arrow[r,hook] &\Grpos
		\end{tikzcd}.
	\end{equation*}

	\subsection{A bit of representation theory of the spin group}
	
	In this section, we give a brief background on the spin group, mostly following \cite[Chapter 20]{FH} and \cite[Section 5]{HenriquesSpeyer}, and prove~\cref{prop:delx} relating Cartan and Pl\"ucker coordinates. Consider the nondegenerate symmetric bilinear form $Q:\C^{2n} \times \C^{2n} \ra \C$ defined by
	\[
	Q(x,y):=\frac 1 2 \sum_{i=1}^{n} (-1)^{i-1} (x_i y_{n+i} + x_{n+i} y_i).
	\]
	
	We first make a change of basis so that $Q$ becomes the standard nondegenerate symmetric bilinear form. Let $W$ denote the Lagrangian subspace $\Span (e_1, e_2,\dots,e_n)$. We have an isomorphism 
	\begin{align}
		W^\perp &\ra W^\vee \nonumber \\
		e_{n+i} &\mapsto  (-1)^{i-1}  e_{i}^\vee, \label{eq:isom}
	\end{align}
	where $W^\vee$ denotes the dual vector space of $W$ and $e_i^\vee$ is basis vector dual to $e_i$, i.e., $e_i^\vee(e_j)=\delta_{ij}$.
	This gives rise to an isomorphism $\C^{2n} \cong W \oplus W^\vee$ such that the inner product $Q$ becomes
	\begin{equation}\label{formqpos}
		Q((x,x^\vee),(y,y^\vee))=\frac 1 2 (x^\vee(y)+y^\vee(x))~\text{where}~(x,x^\vee), (y,y^\vee) \in W \oplus W^\vee.
	\end{equation}
	Note that our form $Q$ agrees with \cite{HenriquesSpeyer} and differs from the standard form in \cite{FH} by a factor of $\frac 1 2$.
	
	Let $\Cl(Q):=\bigoplus_{k=0}^\infty (\C^{2n})^{\otimes k}/\langle x \otimes x - Q(x,x)\rangle$ denote the \textit{Clifford algebra}. Since the ideal $\langle x \otimes x - Q(x,x)\rangle$ is generated by elements of even degree, the Clifford algebra has a $\Z/2\Z$ grading: $\Cl(Q)=\Cl(Q)^{\text{even}} \oplus \Cl(Q)^{\text{odd}}$.
	
	The \textit{Clifford group} \[\operatorname{Cl}^*(Q):=\{x \in \Cl(Q) \mid \text{there exists $y \in \Cl(Q)$ such that $x \otimes y = y \otimes x=1$}\}\]
	is the multiplicative group of units inside $\Cl(Q)$. Its Lie algebra $\cl^*(Q)$ is $\Cl(Q)$ with the Lie bracket $[x,y]:=x \otimes y-y \otimes x$, and we have the exponential map $\exp: \cl^*(Q) \ra \Cl^*(Q)$ defined by
	\begin{equation} \label{exp:cliff}
		\exp(x) := \sum_{n \geq 0} \frac{x^{\otimes n}}{n!}.
	\end{equation}

	The Clifford algebra has an anti-involution $u \mapsto u^*$ called \textit{conjugation} defined by $(x_1 \otimes \cdots \otimes x_r)^*:=(-1)^r x_r \otimes \cdots \otimes x_1.$ The involution $\alpha:\Cl(Q) \ra \Cl(Q)$ defined by $\alpha(x_1 \otimes \cdots \otimes x_r):=(-1)^r (x_1 \otimes \cdots \otimes x_r)$ is called the \textit{main involution}. The \textit{pin} and \textit{spin groups} are defined as  
	\begin{align*}
		\mathrm{Pin}(Q)&:=\{x \in \Cl^*(Q): x \otimes x^* =1 \text{ and }\alpha(x) \otimes \C^{2n} \otimes x^* \subseteq  \C^{2n}\},\\
		\Spin(Q)&:=\{x \in \Cl^*(Q) \cap \Cl(Q)^{\text{even}}: x \otimes x^* =1 \text{ and }\alpha(x) \otimes \C^{2n} \otimes x^* \subseteq  \C^{2n}\}.
	\end{align*}
	The map $\rho:\mathrm{Pin}(Q) \ra \operatorname{O}(Q)$ (resp., $\rho:\Spin(Q) \ra \operatorname{SO}(Q)$) defined by $x \mapsto \rho(x)$ where $\rho(x):\C^{2n} \ra \C^{2n}$ is the endomorphism $v \mapsto \alpha(x) \otimes v \otimes x^*$ makes $\mathrm{Pin}(Q)$ (resp., $\Spin(Q)$) a double cover of $\operatorname{O}(Q)$ (resp., $\operatorname{SO}(Q)$).
	
	The Lie algebra of $\operatorname{SO}(Q)$ is 
	\[\mathfrak{so}(Q):=\{X \in \operatorname{End}(\C^{2n}) \mid Q(X(v),w)+Q(v,X(w))=0~\text{for all $v,w \in \C^{2n}$}\}.\]
	The map $\varphi: \extp^2 \C^{2n} \ra \mathfrak{so}(Q)$ sending $a \wedge b$ to $\varphi_{a \wedge b}$ given by 
	\begin{equation}
		\varphi_{a \wedge b}(v) := 2 (Q(b,v)a-Q(a,v)b) \label{eq:varphi}
	\end{equation}
	is an isomorphism of Lie algebras. On the other hand, the map $\psi:\extp^2 \C^{2n} \ra \cl^*(Q)$ sending $a \wedge b$ to $a \otimes b-Q(a,b)$ is a map of Lie algebras. 
	
	\begin{lemma}\cite[Lemma 20.7 and Exercise 20.33]{FH} \label{fhlemma}
		The composition $\psi \circ \varphi^{-1} : \mathfrak{so}(Q) \ra \Cl(Q)^{\rm{even}}$ is an embedding of Lie algebras. The embedded image is the Lie algebra $\mathfrak{spin}(Q)$ of $\Spin(Q)$.
	\end{lemma}
	
	Let $S:=\extp^\bullet W$. Define the $\Cl(Q)$ representation $\Gamma: \Cl(Q) \ra  \text{End}(S)$ by 
	\begin{align*}
		\Gamma_x (w_1 \wedge \cdots \wedge w_k)&:=x \wedge (w_1 \wedge \cdots \wedge w_k) \text{ for }{x} \in W,\\
		\Gamma_{x^\vee} (w_1 \wedge \cdots  \wedge w_k)&:=x^\vee \lrcorner ( w_1 \wedge \cdots \wedge w_k) \text{ for } x^\vee \in W^\vee,
	\end{align*}
{where $x^\vee \lrcorner ( w_1 \wedge \cdots \wedge w_k):=\sum_{i=1}^k (-1)^{i-1} x^\vee(w_i) w_1 \wedge \cdots \wedge \widehat{w_i} \wedge \cdots \wedge w_k$}. This is an isomorphism $\Cl(Q) \cong \text{End}(S)$. Let $S_+ := \extp^{\text{even}} W$ {and} $S_- :=\extp^{\text{odd}} W$. Restricting $\Gamma$, we obtain an isomorphism 
	\[\Gamma:\Cl(Q)^\text{even} \xrightarrow[]{\cong} \text{End}(S_+) \oplus \text{End}(S_-).\]
	
	The embedding $\Spin(Q) \subset \Cl(Q)^{\text{even}}$ makes $S_\pm$ into $\Spin(Q)$ representations, called \textit{half-spin representations.}
	
	For $j \in [n]$, let $c_j(t):=t e_j \otimes e_{j}^\vee + t^{-1} e_{j}^\vee \otimes e_j$, and for $t=(t_1,\dots,t_n) \in (\C^\times)^n$, let 
	\begin{equation} \label{maxtor}
		c(t) := \prod_{j=1}^n c_j(t_j).
	\end{equation}
	The image of $c$ is the \textit{maximal torus} inside $\Spin(Q)$ and under the covering $\rho:\Spin(Q) \ra \SO(Q)$, we get  $\rho(c(t))=\text{diag}(t_1^2,\dots,t_n^2,t_1^{-2},\dots,t_n^{-2})$ (see \cite[Equation~(23.7)]{FH}; the factor of $\frac 1 2$ disappears due to our convention for $Q$).

	$1$ (resp., $e_1$) is a highest weight vector of $S_+$ (resp. $S_-$) with weight $(-1,-1,\dots,-1)$ (resp., $(1,-1,\dots,-1)$). For $I,I^\vee \subseteq [n]$ such that $\# I + \# I^\vee=n+1$, let $e_{I,I^\vee}$ denote wedge product indexed by $I \sqcup I^\vee$: If $I=\{i_1<i_2<\cdots <i_k\}$ and $I^\vee=\{j_1<j_2<\cdots < j_{n+1-k}\}$, then $e_{I,I^\vee}:=e_{i_1} \wedge \cdots \wedge e_{i_k}\wedge e_{j_1}^\vee \wedge \cdots \wedge e_{j_{n+1-k}}^\vee$. Since $\extp^{n+1}\C^{2n}$ is an irreducible $\Spin(Q)$ representation with highest weight vector $e_{\{1\},[n]}$ with weight $(0,-2,\dots,-2)$ {(where $x \in \Spin(Q)$ acts as the special orthogonal transformation $\rho(x)$)}, $\extp^{n+1}\C^{2n}$ is a direct summand of $S_+ \otimes S_-$. Let $p:\extp^{n+1} \C^{2n} \hookrightarrow S_+ \otimes S_- $ denote the morphism of $\mathrm{Spin}(Q)$ representations sending 
	$e_{1,[n]}$ to $(-1)^{ \sum_{j \in [n]} (j-1)} 1 \otimes e_1$.
	Let $\sigma(I)$ be $-1$ if $\# I \equiv 2~\text{modulo}~4$ and $1$ otherwise.
	
	\begin{proposition}\label{prop:delx}
		Suppose $I,I^\vee \subseteq [n]$ are such that $\# I + \# I^\vee=n+1$ and $ I \cap I^\vee = \{l\}$. Then, 
		\begin{align}
			(-1)^{\sum_{j \in I^\vee } (j-1)} p(e_{I,I^\vee})	&= \sigma(I) \sigma(I \setminus \{l\})e_{I} \otimes e_{I \setminus \{l\}}~\text{ if $\#I$ is even, and } \nonumber\\
			(-1)^{\sum_{j \in I^\vee } (j-1)} p(e_{I,I^\vee})	&=\sigma(I) \sigma(I \setminus \{l\})	e_{I \setminus \{l\}} \otimes e_I ~\text{ if $\#I$ is odd.} \label{ptoeii}
		\end{align}
	\end{proposition}
	
	\begin{proof}
		We will use the action of $\Spin(Q)$ to send $e_{1,[n]}$ to $e_{I,I^\vee}$ and use $\Spin(Q)$ equivariance of $p$. The main difficulty will be in keeping track of the signs.
		
		We start by defining the required elements of $\Spin(Q)$. By the Cartan–Dieudonné theorem \cite[Theorem 2.7]{spingeometry}, any element of $\operatorname{O}(Q)$ can be written as a product of reflections, so we look for appropriate reflections. If $w \in V$ with $Q(w,w)=-1$ and $R_w$ is the reflection in the hyperplane orthogonal to $w$, then $w \in \Pin(Q)$ and $\rho(w)=R_w$. Let $u_{jk}:=\frac i {\sqrt{2}} (e_j-e_k + e_j^\vee-e_{k}^\vee)$ and $v_{jk}:=\frac 1 {\sqrt{2}} (e_j-e_k - e_j^\vee+e_k^\vee)$ so that $Q(u_{jk},u_{jk})=Q(v_{jk},v_{jk})=-1$. A computation shows that the composition $R_{v_{jk}} \circ R_{u_{jk}}$ is in $\operatorname{SO}(Q)$ and is the transformation $e_j \longleftrightarrow e_k, e_j^\vee \longleftrightarrow e_k^\vee.
		$ Let $w_j:= e_j - e_j^\vee$, so that $Q(w_j,w_j)=-1$. The composition $R_{w_j} \circ R_{w_k} \in \operatorname{SO}(Q)$ is the transformation $e_j \longleftrightarrow e_j^\vee, e_k \longleftrightarrow e_k^\vee$. The transformations $R_{v_{jk}} \circ R_{u_{jk}}$ and $R_{w_j} \circ R_{w_k}$ have lifts $v_{jk} \otimes u_{jk}$ and $ w_j \otimes w_k$ to $\Spin(Q)$ respectively.
		
		Now, we proceed by induction on $m := \# I$. When $m=1$, we have $I=\{l\}$ and $I^\vee = [n]$ for some $l \in [n]$. Suppose $l \neq 1$. {Since $v_{1l} \otimes u_{1l}$ acts on $\extp^{n+1} \C^{2n}$ as the special orthogonal transformation $\rho(v_{1l} \otimes u_{1l})= R_{v_{1l}} \circ R_{u_{1l}}$,} we have
		\begin{align*}
			v_{1l} \otimes u_{1l} \cdot e_{1,[n]} &=
			R_{v_{1l}} \circ R_{u_{1l}} (e_{1,[n]})\\
			&= e_l \wedge e_l^\vee \wedge e_2^\vee \wedge \cdots \wedge e_{l-1}^\vee \wedge e_1^\vee \wedge e_{l+1}^\vee \wedge \cdots \wedge e_n^\vee\\
			&=-e_{l,[n]},
		\end{align*}
		where the $-1$ arises when we reorder the alternating tensor. Next, we compute the action on $S_+ \otimes S_-$. {We have 
	\begin{align*}
	v_{1l} \otimes u_{1l} \cdot 1&=\Gamma_{v_{1l}} \circ \Gamma_{u_{1l}}(1)\\
	&= \Gamma_{v_{1l}} \circ \frac{i}{\sqrt{2}} (\Gamma_{e_1} -\Gamma_{e_l}+\Gamma_{e_1^\vee}-\Gamma_{e_l^\vee}) (1)\\
	&=\frac{i}{\sqrt 2} \Gamma_{v_{1l}}(e_1 \wedge 1-e_l \wedge 1+e_1^\vee(1)-e_l^\vee(1))\\
	&=\frac{i}{\sqrt 2} \Gamma_{v_{1l}}(e_1-e_l)\\
	&=\frac{i}{2} (\Gamma_{e_1} -\Gamma_{e_l} -\Gamma_{e_1^\vee} + \Gamma_{e_l^\vee}) (e_1 -e_l)\\
&= \frac i 2 (e_1 \wedge e_1 - e_1 \wedge e_l-e_l \wedge e_1 + e_l \wedge e_l -e_1^\vee(e_1)+e_1^\vee(e_l)+e_l^\vee(e_1)-e_l^\vee(e_l))\\
	&= -i 1,
\end{align*}	}	
		and similarly,
		\begin{align*}
			v_{1l} \otimes u_{1l} \cdot e_1 &= \Gamma_{v_{1l}} \left( \frac{i}{\sqrt{2}} (1-e_l \wedge e_1)\right)\\
			&=-i e_l.
		\end{align*}
		Therefore, $v_{1l} \otimes u_{1l} \cdot (1 \otimes e_1) = - 1  \otimes e_l$. By $\Spin(Q)$ equivariance of $p$, for all $l \in [n]$, we have
		\begin{align*}
			p (e_{l,[n]})&= p(-v_{1l} \otimes u_{1l} \cdot e_{1,[n]})\\&=-v_{1l} \otimes u_{1l} \cdot p (e_{1,[n]})\\
			&= -v_{1l} \otimes u_{1l} \cdot (-1)^{ \sum_{j \in [n]} (j-1)}  1 \otimes e_1\\
			&=  (-1)^{ \sum_{j \in [n]} (j-1)} 1 \otimes e_l.
		\end{align*}
		Since $\sigma(I\setminus \{l\})=\sigma(\varnothing)=1$ and $\sigma(I)=\sigma(\{l\})=1$, we get (\ref{ptoeii}) for this case. 
		
		Now suppose $m =\#I> 1$. Let $k$ be the largest element of $I \setminus \{l\}$. {Define $I_0 := I \setminus \{k\}$ and $I_0^\vee:=I \cup \{k\}$ so that we have $\# I_0 + \# I_0^\vee=n+1$ and $I_0 \cap I_0^\vee = \{l\}$}. By a careful computation, we obtain 
		\begin{align}
			R_{{w_k}} \otimes R_{w_l} (e_{I_0,I_0^\vee})&=(-1)^{(k-1)+m} e_{I,I^\vee}, \nonumber \\
			w_{k} \otimes w_{l} \cdot e_{I_0 \setminus \{l\}}&=(-1)^{m-1+\#\{ j \in I_0 \mid j < l\}} e_{I}, \nonumber \\
			w_{k} \otimes w_{l} \cdot e_{I_0}&=(-1)^{m-2+\#\{ j \in I_0 \mid j < l\}} e_{I \setminus\{l\}}. \label{eq:Rwk}
		\end{align}
		 {	
		Since $I_0^\vee = I^\vee \cup \{k\}$, we have 
		\begin{equation} \label{eq:sign:1}
		(-1)^{\sum_{j \in I_0^\vee } (j-1) } = (-1)^{\sum_{j \in I^\vee  } (j-1) + (k-1)}.
		\end{equation}
Since $\# (I_0 \setminus \{l\}) = \# I-2$, we have
\[
\{ \#(I_0 \setminus \{l\})~\text{modulo}~4, \#I~\text{modulo}~4\} =\begin{cases}
	\{0,2\}~\text{if $m$ is even, and}\\
	\{1,3\}~\text{if $m$ is odd},
\end{cases}
\] 
which implies that $ \sigma(I_0 \setminus \{l\}) \sigma(I)=(-1)^{m+1}$. Using this and $\# I_0 = \#(I \setminus \{l\})$, we get \begin{equation}\label{eq:sign:2}
	\sigma(I_0)\sigma(I_0 \setminus \{l\}) \sigma(I) \sigma(I \setminus \{l\}) = \sigma(I_0 \setminus \{l\}) \sigma(I)=(-1)^{m+1}.
	\end{equation} 
	
}

		Assume $m=\#I$ is even so that $\#I_0$ is odd. By the induction hypothesis, 
		\[
		(-1)^{\sum_{j \in I_0^\vee } (j-1)} p(e_{I_0,I_0^\vee})	=\sigma(I_0) \sigma(I_0 \setminus \{l\})	e_{I_0 \setminus \{l\}} \otimes e_{I_0}. 
		\]
		By (\ref{eq:Rwk}) and $\Spin(Q)$ equivariance of $p$,
		\[
		(-1)^{\sum_{j \in I_0^\vee } (j-1)} p((-1)^{(k-1)+m} e_{I,I^\vee})= -\sigma(I_0) \sigma(I_0\setminus\{l\}) e_I \otimes e_{I \setminus \{l\}}.	
		\]
Using (\ref{eq:sign:1}) and (\ref{eq:sign:2}), we get $(-1)^{\sum_{j \in I^\vee } (j-1)} p(e_{I,I^\vee})	= \sigma(I) \sigma(I \setminus \{l\})e_{I} \otimes e_{I \setminus \{l\}}$. The case when $\#I$ is odd is almost identical.
	\end{proof}
	
	\subsection{The positive decorated orthogonal Grassmannian} \label{sec:og}
	
	In this section, we define the orthogonal Grassmannian and its Cartan embedding; for further background, see \cite{chevalley, bkp, HenriquesSpeyer}.
	
	For a subspace $U$ of $V$, let $U^\perp:=\{ {x} \in V \mid  Q(x, y)=0 \text{ for every }y \in U\}$ denote its orthogonal complement. A subspace $U$ is said to be \textit{isotropic} (resp., \textit{coisotropic}) for $Q$ if $U \subseteq U^\perp$ (resp., $U^\perp \subseteq U$). Let $\operatorname{OG}(n,2n)$ denote the orthogonal Grassmannian of isotropic $n$ dimensional subspaces. Then $\operatorname{OG}(n,2n)=\operatorname{OG}_+(n,2n) \sqcup \operatorname{OG}_-(n,2n)$ has two irreducible components, where $\operatorname{OG}_+(n,2n)$ (resp., $\operatorname{OG}_-(n,2n)$) is the $\Spin(Q)$ orbit of $\Span(e_{n+1},\dots,e_{2n})$ (resp., $\Span(e_1,e_{n+2},e_{n+3},\dots,e_{2n})$).
	We have $\Spin(Q)$ equivariant embeddings $\operatorname{Ca}_\pm:\operatorname{OG}_{\pm}(n,2n) \hookrightarrow \mathbb P(S_\pm)$, called \textit{Cartan embeddings}, defined by 
	\[
	\Span(e_{n+1},\dots,e_{2n}) \mapsto {[1]} ~\text{and}~\Span(e_1,e_{n+2},e_{n+3},\dots,e_{2n}) \mapsto {[e_1]}~\text{respectively},
	\]
{where as usual, $[x]$ denotes the projectivization of $x$.} Let $\OG$ denote the {orthogonal Grassmannian} of coisotropic $(n+1)$-dimensional subspaces. Given $X \in \OG$, there are two maximal isotropic subspaces $X_{\pm} \in \operatorname{OG}_{\pm}(n,2n)$ contained in $X$. The composition
	\[
	\begin{tikzcd}[contains/.style = {draw=none,"\in" description,sloped}]
		\OG \arrow[r,hook] &\operatorname{OG}_+(n,2n) \times \operatorname{OG}_-(n,2n) \arrow[r,hook,"{\operatorname{Ca}_+\times \operatorname{Ca}_-}"] &\mathbb P(S_+) \times \mathbb P(S_-) \\
		X \ar[u,contains] \ar[r,mapsto] &(X_+,X_-)\arrow[u,contains] \arrow[r,mapsto]& (\text{Ca}_+(X_+),\text{Ca}_-(X_-))\arrow[u,contains]
	\end{tikzcd}
	\]
	defines a $\Spin(Q)$ equivariant embedding $\operatorname{Ca}:\OG \hookrightarrow \mathbb P(S_+) \times \mathbb P(S_-)$. Let \[\widetilde{\operatorname{OG}}(n+1,2n):=\{(X,s_+,s_-) \mid X \in \OG, s_{\pm} \in \operatorname{Ca}_\pm(X_{\pm})\}\] denote the \textit{decorated orthogonal Grassmannian}. Then, we have an embedding $\OGdec \hookrightarrow S_+ \times S_-$ sending $(X,s_+,s_-)$ to $(s_+,s_-)$.

	Recall that $\sigma(I)$ is defined to be $-1$ if $\# I \equiv 2~\text{modulo}~4$ and $1$ otherwise.
	The coefficients $\Sigma_I(X,s_+,s_-)$ of $\sigma(I)e_I$ in $(s_+,s_-)$ are called \textit{Cartan coordinates}. Consider the bihomogeneous equations
	\begin{equation} \label{cuberecsign}
		\Sigma_{I \cup \{j,l\}} \Sigma_{I \cup \{k\}}=\Sigma_I \Sigma_{I \cup\{j,k,l\}} + \Sigma_{I \cup \{j,k\}} \Sigma_{I \cup\{l\}} + \Sigma_{I \cup \{k,l\}} \Sigma_{I \cup \{j\}},
	\end{equation}
	for $j<k<l$.
	\begin{theorem}[Henriques and Speyer, {\cite[Theorem 5.3]{HenriquesSpeyer}}] \label{thm:HSeqns}
		The image of $\OGdec$ in $S_+ \times S_-$ is the subvariety cut out by all the equations (\ref{cuberecsign}).
	\end{theorem}

{

\begin{remark}
The actual statement of \cite[Theorem 5.3]{HenriquesSpeyer} is that the image of $\OG$ in $\mathbb P(S_+)\times \mathbb P(S_-)$ is the closed subvariety defined by the bihomogeneous equations (\ref{cuberecsign}), but this implies Theorem~\ref{thm:HSeqns} because $(s_+,s_-) \neq (0,0)$ is in the image of $\OGdec$ in $S_+ \times S_-$ if and only if $([s_+],[s_-])$ is in the image of $\OG$ in $\mathbb P(S_+)\times \mathbb P(S_-)$. 
\end{remark}

}
	
	Consider the $\Spin(Q)$ equivariant map $\eta:\OGdec \ra \widetilde{\operatorname{Gr}}(n+1,2n)$ defined by $(\Span(e_1,e_1^\vee,\dots,e_n^\vee),1,e_1) \mapsto (\Span(e_1,e_1^\vee,\dots,e_n^\vee),e_{1,[n]})$.

	\begin{remark}
		The maps $\eta:\OGdec \ra \widetilde{\operatorname{Gr}}(n+1,2n)$ and $S_+ \times S_- \ra S_+ \otimes S_-$ are not embeddings but they become embeddings upon projectivization. 
	\end{remark}

	Given $I \in \binom{[2n]}{n+1}$, let 
	\begin{equation} \label{eq:jjdual}
		J:=I \cap [n]~\text{and}~  J^\vee:=\{i-n \mid i \in I \cap [n+1,2n]\}.
	\end{equation}
	Under the change of basis (\ref{eq:isom}), $e_I$ becomes $(-1)^{\sum_{j \in J^\vee} (j-1)}e_{J,J^\vee}$. The following proposition relates Pl\"ucker and Cartan coordinates.
	\begin{proposition} \label{prop:atob}
		Let $(X,s_+,s_-) \in \OGdec$, let $(X,\vect)= \eta(X,s_+,s_-)$, and let $J,J^\vee$ be defined as in (\ref{eq:jjdual}). If $\#(J \cap J^\vee)=1$, then \[\Delta_I(X,\vect) = \Sigma_J(X,s_+,s_-) \Sigma_{[n] \setminus J^\vee}(X,s_+,s_-).\]
	\end{proposition}
	
	\begin{proof}
		Consider the following commutative diagram 
		\[
		\begin{tikzcd}
			\OGdec \arrow[d,"\eta"'] \arrow[rr,hook] && S_+ \times S_- \arrow[d]\\
			\widetilde{\operatorname{Gr}}(n+1,2n) \arrow[r,hook] & \extp^{n+1}\C^{2n} \arrow[r,hook,"p"] & S_+ \otimes S_-
		\end{tikzcd}.
		\]
		Let $(X,s_+,s_-) \in \OGdec$ and let $(X,\vect) = \eta(X_,s_+,s_-)$. The coefficient of $ \sigma(J) \sigma([n] \setminus J^\vee)e_J \otimes e_{[n] \setminus J^\vee}$ in $s_+ \otimes s_-$ is $\Sigma_J(X,s_+,s_-) \Sigma_{[n] \setminus J^\vee}(X,s_+,s_-)$. Using \cref{prop:delx}, and commutativity of the diagram, we get that this coefficient is also equal to $\Delta_{I}(X,\vect)$.

	\end{proof}
	
	\begin{definition}\label{def:og}
		Let $\OGpos$ denote the subset of $\OGdec$ where all the Cartan coordinates are positive, which we call the \textit{positive decorated orthogonal Grassmannian}.
		Let $\operatorname{OG}_{>0}(n+1,2n)$ denote the \textit{positive orthogonal Grassmannian}, the image of $\OGpos$ under the projection $\OGdec \ra \OG$, or equivalently, the subset of $\operatorname{OG}(n+1,2n)$ where the ratio of any two Cartan coordinates of the same parity is positive.
	\end{definition}

	\begin{example} \label{example:matrix}
		Given $(X,s_+,s_-) \in \OGpos$, \cref{prop:atob} lets us write down a matrix whose row span is $X$. For example, let $n=3$ and let $(X,s_+,s_-) \in \widetilde{\operatorname{OG}}_{>0}(4,6)$ be such that $(\Sigma_J(X,s_+,s_-))_{J \subseteq [3]}=(\Sigma_J)_{J \subseteq [3]}$. Then,
		\[
		X = \text{row span}		\begin{bmatrix}
			\Sigma_\varnothing \Sigma_1 & {\Sigma_\varnothing \Sigma_2} & \Sigma_\varnothing \Sigma_3 & 0&0&0 \\
			0 & \frac{\Sigma_{12}}{\Sigma_\varnothing }& \frac{\Sigma_{13}}{\Sigma_\varnothing }&1&0&0\\
			0& - \frac{\Sigma_{12}\Sigma_2}{\Sigma_\varnothing \Sigma_1 }&-\frac{\Sigma_{12} \Sigma_{23}+\Sigma_{12}\Sigma_3}{\Sigma_\varnothing \Sigma_1}&0&1&0\\
			0&\frac{\Sigma_\varnothing \Sigma_{13}+\Sigma_{12}\Sigma_3}{\Sigma_\varnothing \Sigma_1}&\frac{\Sigma_{13}\Sigma_3}{\Sigma_\varnothing \Sigma_1}&0&0&1
		\end{bmatrix},
		\]
		where $\Sigma_2 = \frac{\Sigma_\varnothing \Sigma_{123}+\Sigma_1 \Sigma_{23}+\Sigma_{12}\Sigma_3}{\Sigma_{13}}$.
	\end{example}

	\subsection{Pfaffian formulas for Cartan coordinates} \label{sec:pfformulas}
	The main result of this section is~\cref{prop:pfaffian} expressing each Cartan coordinate as the Pfaffian of a certain matrix. Let $A=(a_{ij})$ be a $2n \times 2n$ skew symmetric matrix. Let $\omega_A:=\sum_{1 \leq i<j \leq 2n } a_{ij} e_i \wedge e_j$ denote the associated alternating form. The pfaffian $\pf(A)$ of $A$ is defined by the formula 
	\[
	\frac{1}{n!}\omega_A^{\wedge n}=\pf(A) e_1 \wedge \cdots \wedge e_{2n},
	\]
	{where $\omega_A^{\wedge n} $ denotes the wedge product of $n$ copies of $\omega_A$}. For $I \subseteq [2n]$, let $A_I^I$ denote the principal submatrix of $A$ with rows and columns indexed by $I$. 
	\begin{lemma}[{\cite[Chapter 5, Equation (3.6.3)]{procesi}}]\label{lem:exppfaffian}
		We have 
		\[
		\exp({\omega_A})=\sum_{I \subseteq [2n] \mid \# I~\text{is even}} \pf (A_{I}^I) e_I.
		\]
	\end{lemma}
	
	Recall from~\cref{sec:og} that the orthogonal Grassmannian $\operatorname{OG}(n,2n) = \operatorname{OG}_+(n,2n) \sqcup \operatorname{OG}_-(n,2n)$ is the union of two components. If $X_+ \in \operatorname{OG}_+(n,2n)$ and $\Delta_{[n+1,2n]}(X_+) \neq 0$, then in the coordinates (\ref{eq:isom}), $X_+$ is the row span of a matrix of the form $\begin{bmatrix}
		M_+& I_{n}
	\end{bmatrix}$, where $M_+$ is a skew symmetric $n \times n$ matrix. Similarly, if $X_- \in \operatorname{OG}_-(n,2n)$ is such that $\Delta_{\{1,n+2,n+3,\dots,2n\}}(X_-) \neq 0$, then $X_-$ is the row span of an $n \times 2n$ matrix with $I_n$ in columns $1,n+2,n+3,\dots,2n$ and a matrix $\tilde M_-$ in columns $2,3,\dots,n+1$ such that the matrix $M_-$ obtained from $\tilde M_-$ by cyclically rotating the columns by one step to the right is skew symmetric. For $J \subseteq [n]$, let $J \Delta \{1\}$ denote the symmetric difference, i.e., 
	\[
	J \Delta \{1\} := \begin{cases}
		{J \setminus \{1\}}&\text{if $1 \in J$};\\
		{ J \cup \{1\}}&\text{if $1 \notin J$}.
	\end{cases}
	\]

\begin{lemma} \label{lemma:exchange}{Let $w:=e_1 - e_1^\vee \in \Pin(Q)$ so that $\rho(w)=R_w \in \operatorname O(Q)$ is given by $e_{1} \longleftrightarrow e_{n+1}$. The following diagram commutes:}
	\[
	\begin{tikzcd} \operatorname{OG}_+(n+1,2n)  \arrow[r,hook,"\operatorname{Ca}_+"]\arrow[d,"R_w"',"\cong"] & \mathbb P(S_+) \arrow[d,"\Gamma_w","\cong"']\\ \operatorname{OG}_-(n+1,2n)\arrow[r,hook,"\operatorname{Ca}_-"] &\mathbb P(S_-)  
	\end{tikzcd}.
	\]
\end{lemma}
\begin{proof}
	{
	Let $X \in \operatorname{OG}_+(n+1,2n)$, so $X = \rho(x) \cdot \Span(e_{n+1},\dots,e_{2n})$ for some $x \in \Spin(Q)$. By $\Spin(Q)$ equivariance of $\operatorname{Ca}_+$, we have 
	\[
\Gamma_w(\operatorname{Ca}_+(X)) = [\Gamma_w \circ \Gamma_x(1)] = [\Gamma_{w \otimes x \otimes w^*} \circ \Gamma_w(1)]=[\Gamma_{w \otimes x \otimes w^*} (e_1)].	
	\]
On the other hand, noting that $w \otimes x \otimes w^* \in \Spin(Q)$ and using $\Spin(Q)$ equivariance of $\operatorname{Ca}_-$, we get
\begin{align*}
\operatorname{Ca}_- ( R_w \cdot X)&=\operatorname{Ca}_-(\rho(w \otimes x)\cdot \Span(e_{n+1},\dots,e_{2n}))\\
&= \operatorname{Ca}_-(\rho(w \otimes x \otimes w^*) \rho(w) \cdot \Span(e_{n+1},\dots,e_{2n}))\\&= 
\operatorname{Ca}_-(\rho(w \otimes x \otimes w^*)\cdot  \Span(e_{1},e_{n+2},\dots,e_{2n}))\\&=[\Gamma_{w \otimes x \otimes w^*} (e_1)]\\&=\Gamma_w(\operatorname{Ca}_+(X)).
\end{align*}
}
\end{proof}

	\begin{proposition} \label{prop:pfaffian}
		Let $(X,s_+,s_-) \in \OGpos$ and let $(X_+,X_-)$ denote the maximal isotropic subspaces in $X$. Let $M_+, M_-$ be as above. Then,
		\[
		\Sigma_J(X,s_+,s_-) = \begin{cases}
			\Sigma_\varnothing(X,s_+,s_-) \pf ((M_{+})^J_J) &\text{if $\#J$ is even};\\
			(-1)^{\frac{\#J-1}{2}}\Sigma_{\{1\}}(X,s_+,s_-) \pf ((M_{-})^{J \Delta \{1\}}_{J \Delta \{1\}}) &\text{if $\#J$ is odd}.
		\end{cases}
		\]
	\end{proposition}
	\begin{proof}
		We use $\Spin(Q)$ equivariance of the Cartan map and the following commutative diagram of exponential maps:
		\[
		\begin{tikzcd} \mathfrak{so}(Q)  \arrow[r,"\psi \circ \varphi^{-1}","\sim"']\arrow[d,"\exp"'] & \mathfrak{spin}(Q) \arrow[d,"\exp"]\\ \operatorname{SO}(Q) &\Spin(Q)  \arrow[l,"\rho"']
		\end{tikzcd},
		\]
		where the $\exp$ on the left is the matrix exponential map and on the right is (\ref{exp:cliff}).

		Consider the element ${m_+:=\begin{bmatrix}
				0 & -M_+\\
				0 & 0
			\end{bmatrix}
		} \in \mathfrak{so}(Q)$. Exponentiating $m_+$, we get ${\begin{bmatrix}
				I_n & -M_+ \\
				0 & I_n
			\end{bmatrix}
		} \in \SO(Q)$, so that we have $\begin{bmatrix}
			0&I_n
		\end{bmatrix}(\exp{(m_+)})^T=\begin{bmatrix}
			M_+ &I_n
		\end{bmatrix}$. On the other hand, under the isomorphism $\psi \circ \varphi^{-1}:\mathfrak{so}(Q) \xrightarrow{\sim} \mathfrak{spin}(Q)$, $m_+$ goes to $-\sum_{1\leq i<j \leq n} (M_+)_{ij} e_i \otimes e_j$. Exponentiating, and using Lemma \ref{lem:exppfaffian} along with $\Spin(Q)$ equivariance of the Cartan map, we get
		\begin{align} \label{eq:pfeven}
			[s_+] =[\exp(m_+) \cdot 1] &=\left[ \sum_{J~\text{even}} \pf ((-M^{+})^J_J) e_J \right].
		\end{align}
		Since $\Sigma_J(X,s_+,s_-)$ is the coefficient of $\sigma(J)e_J$ in $s_+$, we get 
		\begin{align*}
			\frac{	\Sigma_J(X,s_+,s_-)}{\Sigma_\varnothing(X,s_+,s_-)}=\sigma(J)  \pf ((-M^{+})^J_J)= \pf ((M^{+})^J_J),
		\end{align*}
		where we used $\pf ((-M^{+})^J_J)=(-1)^{\frac{\#J}{2}}\pf ((M^{+})^J_J)$ and $(-1)^{\frac{\#J}{2}} = \sigma(J)$.
		
	{	
		Let $w:=e_1 - e_1^\vee \in \Pin(Q)$ as in Lemma~\ref{lemma:exchange}. Since $\rho(w)=R_w \in \operatorname O(Q)$ is given by $e_{1} \longleftrightarrow e_{n+1}$, it sends $X_-$ to $\text{row span}\begin{bmatrix}
			M_- &I_n
		\end{bmatrix} \in \operatorname{OG}_+(n+1,2n)$. Moreover, $\Gamma_w:S_+ \ra S_-$ is given by
		\[
		\Gamma_w(e_J) = \begin{cases}
			-e_{J \Delta \{1\}}&\text{if $1 \in J$};\\
			e_{ J \Delta \{1\}}&\text{if $1 \notin J$}.
		\end{cases}
		\]
	Therefore, by Lemma~\ref{lemma:exchange}, we get	}
		\[
		\left[-	\sum_{J~\text{odd}\mid 1 \in J} \Sigma_J(X,s_+,s_-) e_{J \Delta \{1\}} + \sum_{J~\text{odd}\mid 1 \notin J} \Sigma_J(X,s_+,s_-) e_{J \Delta \{1\}}\right]=
		\left[\sum_{J~\text{odd}} \pf ((-M_{-})^{J \Delta \{1\}}_{J \Delta \{1\}}) e_{J \Delta \{1\}}\right].	
		\]
		Now, we have to check two cases. If $1 \in J$, then 
		\[
		\frac{\Sigma_J(X,s_+,s_-)}{\Sigma_{\{1\}}(X,s_+,s_-)}= (-1)^{\frac{\#J-1}{2}}\pf ((M_{-})^{J \Delta \{1\}}_{J \Delta \{1\}}),
		\]
		and if $1 \notin J$, then 		
		\[
		\frac{\Sigma_J(X,s_+,s_-)}{\Sigma_{\{1\}}(X,s_+,s_-)}=- (-1)^{\frac{\#J+1}{2}}\pf ((M_{-})^{J \Delta \{1\}}_{J \Delta \{1\}})=(-1)^{\frac{\#J-1}{2}}\pf ((M_{-})^{J \Delta \{1\}}_{J \Delta \{1\}}).
		\]
	\end{proof}

	\begin{example}
		Recall~\cref{example:matrix}.  After making the change of basis (\ref{eq:isom}), the two maximal isotropic subspaces $X_+$ and $X_-$ are the row spans of 
		\[
		\begin{bmatrix}0&\frac{\Sigma_{12}}{\Sigma_\varnothing }&\frac{\Sigma_{13}}{\Sigma_\varnothing }&1&0&0\\
			-\frac{\Sigma_{12}}{\Sigma_\varnothing }&0&\frac{\Sigma_{23}}{\Sigma_\varnothing }&0&1&0\\
			-\frac{\Sigma_{13}}{\Sigma_\varnothing }&-\frac{\Sigma_{23}}{\Sigma_\varnothing }&0&0&0&1
		\end{bmatrix}~\text{and}~\begin{bmatrix}
			1& \frac{\Sigma_2}{\Sigma_1} & \frac{\Sigma_3}{\Sigma_1}&0&0&0\\
			0&0&-\frac{\Sigma_{123}}{\Sigma_1} &  -\frac{\Sigma_{2}}{{\Sigma_1}}&1&0\\
			0&\frac{\Sigma_{123}}{\Sigma_1} &0& -\frac{\Sigma_{3}}{\Sigma_1}&0&1
		\end{bmatrix}~\text{respectively}.
		\]
		Therefore,
		\[
		M_+=	\begin{bmatrix}0&\frac{\Sigma_{12}}{\Sigma_\varnothing }&\frac{\Sigma_{13}}{\Sigma_\varnothing }\\
			-\frac{\Sigma_{12}}{\Sigma_\varnothing }&0&\frac{\Sigma_{23}}{\Sigma_\varnothing }\\
			-\frac{\Sigma_{13}}{\Sigma_\varnothing }&-\frac{\Sigma_{23}}{\Sigma_\varnothing }&0
		\end{bmatrix}~\text{and}~M_-=\begin{bmatrix}
			0& \frac{\Sigma_2}{\Sigma_1} & \frac{\Sigma_3}{\Sigma_1}\\
			-\frac{\Sigma_{2}}{{\Sigma_1}}&0&-\frac{\Sigma_{123}}{\Sigma_1}\\
			-\frac{\Sigma_{3}}{\Sigma_1}&\frac{\Sigma_{123}}{\Sigma_1} &0& 
		\end{bmatrix},
		\]
		using which we verify~\cref{prop:pfaffian}. For example,
		\begin{align*}
			\Sigma_\varnothing \pf ((M_{+})^{12}_{12})&=\Sigma_\varnothing \pf \begin{bmatrix}
				0 & \frac{\Sigma_{12}}{\Sigma_\varnothing}\\
				-\frac{\Sigma_{12}}{\Sigma_\varnothing}&0
			\end{bmatrix}=\Sigma_{12}~\text{when $J=\{1,2\}$},\\
			(-1)^{\frac{1-1}{2}}	\Sigma_1 \pf ((M_{-})^{12}_{12})&=\Sigma_1 \pf \begin{bmatrix}
				0 & \frac{\Sigma_{2}}{\Sigma_1}\\
				-\frac{\Sigma_{2}}{\Sigma_1}&0
			\end{bmatrix}=\Sigma_{2}~\text{when $J=\{2\}$, and}\\
			(-1)^{\frac{2-1}{2}}		 \Sigma_1 \pf ((M_{-})^{23}_{23})&=-\Sigma_1 \pf \begin{bmatrix}
				0 & \frac{-\Sigma_{123}}{\Sigma_1}\\
				\frac{\Sigma_{123}}{\Sigma_1}&0
			\end{bmatrix}=\Sigma_{123}~\text{when $J=\{1,2,3\}$}.
		\end{align*}
		
	\end{example}

	\subsection{\texorpdfstring{$B$}{B} variables} \label{sec:bvar}
	
	Consider the space $\mathcal B_G:=\Rpos^{V(G) \sqcup F(G)}$ of functions $B : V(G) \sqcup F(G) \ra \Rpos$. 
	We call a pair $(G,B)$ a \textit{$B$-network}. Since there is a bijection $V(G) \sqcup F(G) \xrightarrow[]{\sim} W(G_+)$, we will sometimes write $B_{\wv_u}$ instead of $B_u$ for $u \in V(G) \sqcup F(G)$.

	A Y-$\Delta$ move $G \rightsquigarrow G'$ induces a homeomorphism $\mathcal B_{G} \xrightarrow[]{\sim} \mathcal B_{G'}$ given by the \textit{cube recurrence}
	\[
	B_{f_0'} := \frac{B_{v_1}B_{f_1}+B_{v_2}B_{f_2}+B_{v_3}B_{f_3}}{B_{v_0}}
	\]
	and $B_{v'} := B_v$ for all other $v \in V(G')$ and $B_{f'}:=B_f$ for all other $f \in F(G')$, where vertices and faces are labeled as in Figure~\ref{fig:ydmove}. Define $\mathcal B_n := \bigsqcup_{\tau_G=\tau_n} \mathcal B_{G}\Big/\text{moves}$.
	
	\begin{definition} \label{def:atob}
	Each face $g$ of $G_+$ is incident to two white vertices $\wv_v$ and $\wv_f$, where $v \in V(G)$ and $f \in F(G)$. Define the inclusion $i_G^+:\mathcal B_{G} \hookrightarrow \mathcal A_{G_+}$ by $A_{g}:=B_{v} B_{f}$. 
	\end{definition}
	\begin{proposition}[Goncharov and Kenyon, {\cite[Lemma 5.11]{GK}}]
		If $G$ and $G'$ are related by a Y-$\Delta$ 
		move, then there is a sequence of moves relating $G_+$ and $G_+'$ such that the following diagram commutes.
		\[
		\begin{tikzcd} \mathcal B_{G} \arrow[r,hook,"i_G^+"]\arrow[d,"\text{Y-$\Delta$ move}"',"\sim"] & \mathcal A_{G_+}\arrow[d,"\text{moves}","\sim"']\\ \mathcal B_{G'} \arrow[r,hook,"i_{G'}^+"] &\mathcal A_{G_+'}
		\end{tikzcd}
		\]
		Therefore, the inclusions $i_G^+$ glue to an inclusion $i^+:\mathcal B_n \hookrightarrow \mathcal A_{\pi_{n+1,2n}}$.
	\end{proposition}
	\begin{definition}	
		Given $G$ a reduced graph with $\tau_G=\tau_n$, we assign to each vertex and face of $G$ a subset of $[n]$ as follows. For $j \in [n]$, let $\beta_j$ denote {the strand in $G$ between $t_{j}$ and $t_{n+j}$}. {The faces of the medial graph $G^\times$ are in bijection with $V(G) \sqcup F(G)$. We say that $u \in V(G) \sqcup F(G)$ is to the left of $\beta_j$ if the corresponding face of $G^\times$ is to the left of $\beta_j$ when $\beta_j$ is oriented from $t_{n+j}$ to  $t_j$.} For $u \in V(G) \sqcup F(G)$, define
		\[
		J(u):=\{j \in [n] \mid u~\text{is to the left of $\beta_j$}\}.
		\]
	\end{definition}

		\begin{figure}
	\begin{tabular}{ccc}
		\includegraphics[width=0.2\textwidth]{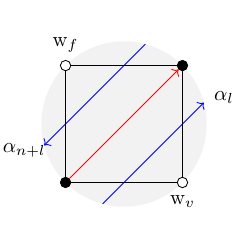} 
	\end{tabular}
	\caption{\label{fig:localfaceg} The strands of $G$ and $G_+$ near a face $g$. The red strand is $\beta_l$ oriented from $t_{n+l}$ to $t_l$.}
\end{figure}

	\begin{lemma}\label{lem:jvjf}
		Let $g$ be a face of $G_+$ incident to white vertices $\wv_v, \wv_f$ where $v \in V(G)$ and $f \in F(G)$. If $I:=S(g)$, then $\{J,[n]\setminus J^\vee\} = \{J(v),J(f)\}$, where $J$ and $J^\vee$ are as in (\ref{eq:jjdual}).
	\end{lemma}
	\begin{proof}
		{
		Let $\alpha_i$ denote the strand in $G_+$ ending at $d_i^+$ as in Section~\ref{sec:gtogplus} so that $\alpha_i$ and $\alpha_{n+i}$ are the two strands in $G_+$ that correspond to the strand $\beta_{i}$ in $G$. Note that the faces of $G_+$ are in bijection with edges of $G^\times$. Let $l \in [n]$ be such that $\beta_l$ is the strand containing the edge of $G^\times$ corresponding to $g$. Without loss of generality, assume that $v$ is to the left and $f$ to the right of $\beta_l$. Then, clearly we have $J(v) = J(f) \cup \{l\}$. Note that $l, n+l \in S(g)$ since $g$ is to the left of both $\alpha_l$ and $\alpha_{n+l}$ (Figure~\ref{fig:localfaceg}). For $i \neq l$, there is exactly one strand in the pair $\{\alpha_i,\alpha_{n+i}\}$ that $g$ is to the left of: if $i \in J(f)$, then this is $\alpha_i$ and if $i \notin J(v)$, then this is $\alpha_{n+i}$. Any $\beta_i, i\neq l$ is of one of these two types (i.e., $[n] \setminus J(v) \sqcup J(f) = [n] \setminus \{l\}$), so we get 
		\[S(g)=\{l,n+l\} \cup \{i  \mid i \in J(f)\} \cup \{n+i \mid i \notin J(v)\}.\]
		Therefore, $J = S(g) \cap [n] = J(f) \cup \{l\}=J(v)$ and $J^\vee = ([n]\setminus J(v))  \cup\{l\} = [n] \setminus J(f)$.}
	\end{proof}

	\begin{figure}
		\begin{tabular}{cc}
			\includegraphics[width=0.3\textwidth]{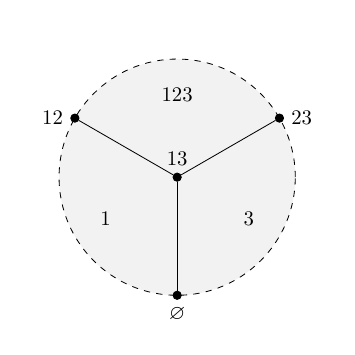} 
			&\includegraphics[width=0.3\textwidth]{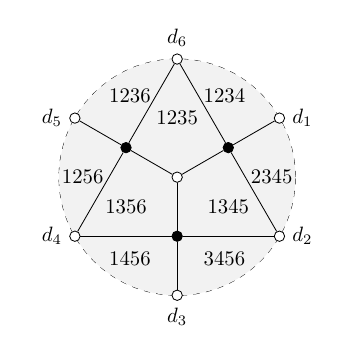}\\
			(a)  &(b) 
		\end{tabular}
		\caption{\label{fig:network3labels} (a) Labeling the vertices and faces of the electrical network from Figure~\ref{fig:facenetwork3}(a) and (b) the face labels of $G_+$.}
	\end{figure}
	\begin{example}
		Figures~\ref{fig:network3labels}(a) and (b) show the labels $J$ and $S$ for $G$ and $G_+$ from Figure~\ref{fig:facenetwork3}.
	\end{example}
	
	Define the map
	\[
	\Psi_G: \OGpos \ra \mathcal B_G
	\]
	sending $(X,s_+,s_-)$ to $(\Sigma_{J(u)}(X,s_+,s_-))_{u \in V(G) \sqcup F(G)}$. 
	
	\begin{theorem}[Henriques and Speyer, \cite{HenriquesSpeyer}] 
		For every reduced $G$ with $\tau_G = \tau_n$, $\Psi_G:\OGpos \xrightarrow{\sim} \mathcal B_{G}$ is a homeomorphism. 
	\end{theorem}

		\begin{figure}[h]
		\begin{tabular}{ccc}
			\includegraphics[width=0.6\textwidth]{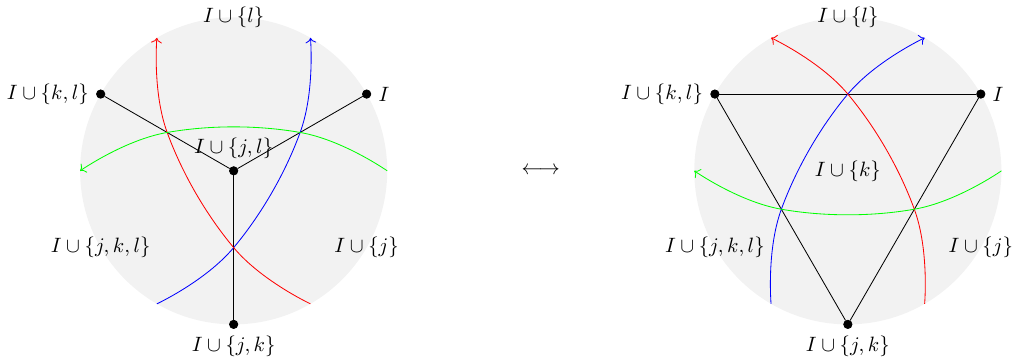} 
		\end{tabular}
		\caption{\label{fig:ydmovezz} {Labeling of the vertices and faces in a Y-$\Delta$ move. The green, red and blue strands are $\beta_j$, $\beta_k$ and $\beta_l$ respectively. The strand $\beta_i$ for $i \in \{j,k,l\}$ is oriented from $t_{n+i}$ to $t_i$. It is a consequence of $G$ being reduced that locally, the orientations of the strands must be as shown here up to a cylic rotation. Indeed, since each pair of strands crosses once inside the small disk $R$ shown in the figure where the Y-$\Delta$ move occurs, they cannot cross again in $\mathbb D \setminus R$. Therefore, the cyclic order of their endpoints around the boundary of $R$ is the same as the cyclic order of their endpoints around the boundary of $\mathbb D$, so we have three consecutive in-endpoints followed by three consecutive out-endpoints. }}
	\end{figure}
	
	Suppose $G$ and $G'$ are related by a Y-$\Delta$ move with vertices and faces labeled as in Figure~\ref{fig:ydmove}. Then, up to cyclic rotation of the tuple $(v_1,f_1,v_2,f_2,v_3,f_3)$, we have
	\begin{align*}
		J(v_0)&=I \cup\{j,l\},& J(v_1)&=I,&J(v_2)&=I \cup\{j,k\},&J(v_3)&=I \cup\{k,l\},\\
		J(f_0)&=I\cup\{k\},& J(f_1)&=I \cup\{j,k,l\},&J(f_2)&=I \cup\{l\},&J(f_3)&=I \cup\{j\}.
	\end{align*}
	for some $I \subset [n]$ and $ j<k<l$ (see Figure~\ref{fig:ydmovezz}). By~\cref{thm:HSeqns}, the following diagram commutes
	\[
	\begin{tikzcd}
		&\mathcal B_{G}\arrow[dd,"\text{Y-$\Delta$ move}","\sim"'] \arrow[dr]&\\ \OGpos\arrow[ur,"\Psi_{G}","\sim"']\arrow[dr,"\Psi_{G'}"',"\sim"]&&\mathcal B_n\\ &\mathcal B_{G'} \arrow[ur]&
	\end{tikzcd},
	\]
	so we obtain a well-defined homeomorphism $\Psi:\OGpos \xrightarrow[]{\sim} \mathcal B_n$.

	\begin{proposition}\label{prop:ogg}
		The following diagram commutes.
		\begin{equation} \label{ogg}
			\begin{tikzcd} 
				\mathcal B_{G} \arrow[r,"i_{G}^+"] &\mathcal A_{G_+}\\
				\OGpos \arrow[r]\arrow[u,"\Psi_G","\sim"'] & \widetilde{\operatorname{Gr}}_{>0}(n+1,2n) \arrow[u,"\Phi_{G_+}"',"\sim"]
			\end{tikzcd}.
		\end{equation}	
	\end{proposition}
	\begin{proof}
		Let $(X,\vect) \in \OGdec$ and let $(X,s_+,s_-)= \eta(X,\vect)$. Let $g$ be a face of $G_+$ incident to white vertices $\wv_v, \wv_f$ where $v \in V(G)$ and $f \in F(G)$. Using~\cref{lem:jvjf} and \cref{prop:atob}, we get
		\begin{equation} \label{a2b}
			\Delta_{S(g)}(X,\vect) = \Sigma_{J(v)}(X,s_+,s_-) \Sigma_{J(f)}(X,s_+,s_-).,
		\end{equation}
		which implies that (\ref{ogg}) commutes. 
	\end{proof}

	\section{The electrical right twist}\label{sec:electwistt}
	
	In this section, we define the electrical right twist and prove~\cref{thm:elecrtintro}. 
	\begin{definition}\label{def:q}
		Let $e=uv$ be an edge of $G$ and let $f,g$ denote the faces of $G$ incident to $e$. Following equation (56) in \cite[Section 5.3.1]{GK}, define $q_G: \mathcal B_G \ra \mathcal R_{G}$ by $c(e):=\frac{B_u B_v}{B_f B_g}$. 	
	\end{definition}

	\begin{proposition}\cite[Section 5.3.2]{GK}
	{	If $G$ and $G'$ are related by a Y-$\Delta$ move, then the following diagram commutes}
	\[
	\begin{tikzcd}
		\mathcal B_G \arrow[r,"q_G"] \arrow[d,"\text{Y-$\Delta$ move}"',"\sim"] & \mathcal R_G \arrow[d,"\text{Y-$\Delta$ move}","\sim"']\\
		\mathcal B_{G'} \arrow[r,"q_{G'}"] & \mathcal R_{G'}
	\end{tikzcd}.
	\]
{	Therefore, the maps $q_G$ glue to a map $q: \mathcal B_n \ra \mathcal R_n$.}
	\end{proposition}

	Recall the action (\ref{def:action}) of $\Rpos^{2n}$ on $\Grpos$ by rescaling columns.	
	\begin{definition} \label{def:elecrt}
		Let $(X,s_+,s_-) \in \OGpos$, and let $t_i:=\frac{\Sigma_{J(d_{i-1})}(X,s_+,s_-)}{\Sigma_{J(d_{i})}(X,s_+,s_-)}$ for $i \in [2n]$. The \textit{electrical right twist} of $(X,s_+,s_-)$, denoted $\vec{\tau}_{\rm elec}(X,s_+,s_-)$, is defined to be $ t \cdot \vec \tau(X) \in \Grpos$.
	\end{definition}

	\begin{theorem} \label{thm:elecrt}
		Let $G$ be a reduced graph with $\tau_G=\tau_n$. The image of $\vec \tau_{\rm elec}$ is contained in $\IGpos$, and the following diagrams commute:
		\[
		\begin{tikzcd} \mathcal B_G   \arrow[r,"q_G"] & \mathcal R_{G} \arrow[d,"\sim"',"\Meas_{G_+} \circ j_{G}^+"]\\ \OGpos \arrow[u,"\Psi_G","\sim"']
			\arrow[r,"\vec\tau_{{\rm elec}}"]
			&\IGpos
		\end{tikzcd}, \begin{tikzcd} \mathcal B_n   \arrow[r,"q"] & \mathcal R_{n} \arrow[d,"\sim"',"\Meas \circ j^+"]\\ \OGpos \arrow[u,"\Psi","\sim"']
			\arrow[r,"\vec\tau_{{\rm elec}}"]
			&\IGpos
		\end{tikzcd}.
		\]
	\end{theorem}

	\begin{proof}
		We will show commutativity of the left diagram by showing that $\Meas_{G_+}^{-1} \circ \vec \tau_{\rm elec} = j_G^+ \circ q_G \circ \Psi_G$. The right diagram is then obtained by gluing.
		
		Define $B := \Psi_G(X,s_+,s_-)$ {, $A:=i_G^+(B)$, and $t_i:=\frac{B_{d_{i-1}}}{B_{d_{i}}}$. We have
		\begin{align}
\Meas_{G_+}^{-1} \circ \vec \tau_{\rm elec}(X,s_+,s_-) &=	\Meas_{G_+}^{-1}(t \cdot \vec \tau(X)) && (\text{Definition}~\ref{def:elecrt}) \nonumber \\
&=  t \cdot 	\Meas_{G_+}^{-1} \circ \vec \tau(X)&&(\text{Lemma}~\ref{lem:tequiv})\nonumber\\
&= 	 t \cdot p_{G_+} \circ  \Phi_{G_+} \circ \cev \tau \circ \vec \tau (X) &&(\text{Theorem}~\ref{thm:MSmain})\nonumber\\
&=  t \cdot p_{G_+} \circ  \Phi_{G_+} (X)&&(\cev \tau \circ \vec \tau = \text{id})\nonumber\\
&=t \cdot p_{G_+}(A). &&(\text{Proposition}~\ref{prop:ogg}) \label{eq::big}
		\end{align}		
}	
		
Let $e=uv$ be an edge of $G$ and let $f,g$ denote the faces of $G$ incident to $e$. From definitions, if $[\wt]:= j_G^+ \circ q_G (B)$, then
		\[
		\wt({\rm b}_e \wv_{x}) = \begin{cases}
			\frac{B_u B_v}{B_f B_g} &\text{if }x \in \{u,v\} \\
			1&\text{if }x \in \{f,g\}
		\end{cases}.
		\] 
		We define a gauge transformation $\tilde g$ by $\tilde g({\rm b}_e):= \frac{1}{B_{u} B_{v}}$ and {
	\[		\tilde g({\rm w}_u):= \begin{cases}B_u^2&~\text{if ${\rm w}_u$ is an internal white vertex, and}\\
		1&~\text{if ${\rm w}_u$ is a boundary white vertex.}
		\end{cases}
		\]
	}
 We have the following cases for edges ${\rm b }_e{\rm w}_x$ in $G_+$.
		\begin{enumerate}
			\item $x = u$. Let $h$ and $h'$ be the two faces of $G_+$ incident to ${\rm b}_e {\rm w}_u$, where $h$ is between $u$ and $f$ and $h'$ is between $u$ and $g$. 
			\begin{enumerate}
				\item $\wv_u$ is an internal vertex of $G_+$. Then, $\Meas_{G_+}^{-1} \circ \vec \tau_{\rm elec}$ assigns weight $
				\frac{1}{A_h A_{h'}} = \frac{1}{B_{u}^2 B_{f} B_{g}}
				$ to $\bv_e \wv_u$. Applying the gauge transformation $\tilde g$, we get 
				\[
				B_{u} B_{v} \frac{1}{B_{u}^2 B_{f} B_{g}} B_{u}^2  = \frac{B_{u}B_{v}}{B_{f}B_{g}}.
				\]
				\item $\wv_u$ is a boundary vertex $d_{2i-1}$ of $G_+$.  $\Meas_{G_+}^{-1} \circ \vec \tau_{\rm elec}$ assigns weight $	\frac{B_{d_{2i-1}}}{B_{d_{2i-2}}}\frac{A_{f_{2i-1}^-}}{A_h A_{h'}} = \frac{1}{B_{f} B_{g}}
				$ to $\bv_e \wv_u$. Applying the gauge transformation $\tilde g$, we get $\frac{B_{u}B_{v}}{B_{f}B_{g}}$.
			\end{enumerate}

			\item $x = f$. Let $h$ and $h'$ be the two faces of $G_+$ incident to ${\rm b}_e {\rm w}_f$, where $h$ is between $u$ and $f$ and $h'$ is between $f$ and $v$. 		
			
			\begin{enumerate}
				\item If ${\rm w}_f$ is an internal vertex of $G_+$, then $\Meas_{G_+}^{-1} \circ \vec \tau_{\rm elec}$ assigns weight
				$			\frac{1}{A_h A_{h'}} = \frac{1}{B_{u} B_{f}^2  B_{v}}$ to ${\rm b}_e {\rm w}_f$. Applying the gauge transformation $\tilde g$, we get 
				\[
				B_{u} B_{v} \frac{1}{B_{u} B_{f}^2 B_{v}} B_{f}^2  = 1.
				\]
				\item If ${\rm w}_f$ is the boundary vertex $d_{2i}$, then $\Meas_{G_+}^{-1} \circ \vec \tau_{\rm elec}$ assigns weight
				$\frac{B_{d_{2i}}}{ B_{d_{2i-1}}}\frac{A_{f_{2i}^-}}{A_h A_{h'}} = \frac{1}{B_{u}  B_{v}}
				$ to ${\rm b}_e {\rm w}_f$. Applying the gauge transformation $\tilde g$, we get $1$.
			\end{enumerate}
		\end{enumerate}
{	
	Finally, $ \vec \tau_{\rm elec}(X,s_+,s_-) =\Meas_{G_+} \circ j_G^+ \circ q_G \circ \Psi_G(X,s_+,s_-) \in \IGpos$ by Theorem~\ref{lamthmmain}.
}	
	\end{proof}

	\begin{corollary} \label{cor:elecsurj}
		The electrical right twist $\vec \tau_{\rm elec}:\OGpos \ra \IGpos$ is surjective.
	\end{corollary}
	\begin{proof}
		\cite[Proposition 4]{KW} shows that $q_G$ is surjective. By \cref{thm:elecrt}, $\vec\tau_{\rm elec}$ is surjective. 
	\end{proof}
	
	\begin{figure}
		\begin{tabular}{ccc}
			\includegraphics[width=0.3\textwidth]{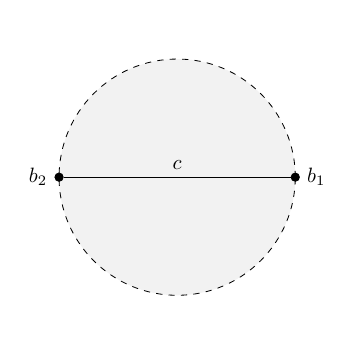} &\includegraphics[width=0.3\textwidth]{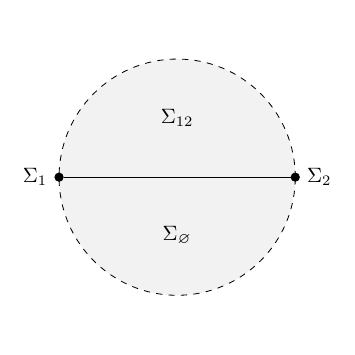}
			&\includegraphics[width=0.3\textwidth]{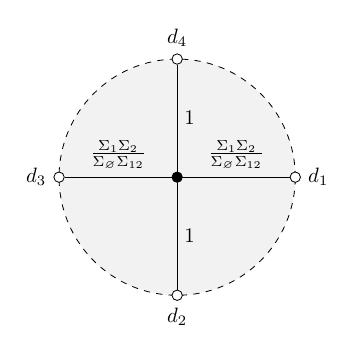}\\
			(a) An electrical network $(G,c)$. &(b) $\Psi_G(X,s_+,s_-)$. &(c) $j_G^+ \circ q_G \circ \Psi_G(X,s_+,s_-)$.
		\end{tabular}
		\caption{\label{fig:facenetwork2} Commutativity of the diagram in~\cref{thm:elecrt} when $n=2$ .}
	\end{figure}

	\begin{example} \label{example:rt2}
		Let $n=2$ and let $(X,s_+,s_-) \in \widetilde{\operatorname{OG}}_{>0}(3,4)$ be such that $(\Sigma_J(X,s_+,s_-))_{J \subseteq [2]}=(\Sigma_J)_{J \subseteq [2]}$. Then, $X$ is the row span of the matrix
		\[
		\begin{bmatrix}
			\Sigma_{\varnothing} \Sigma_1&\Sigma_{\varnothing}\Sigma_{2}&0&0\\
			0& \frac{\Sigma_{12}}{\Sigma_\varnothing}&1&0\\
			0&-\frac{\Sigma_{12} \Sigma_2}{\Sigma_\varnothing \Sigma_1} &0&1
		\end{bmatrix},~\text{so we compute}~\vec \tau_{\rm elec}(X,s_+,s_-)= \begin{bmatrix}
			\frac{\Sigma_{12}}{\Sigma_\varnothing \Sigma_1 \Sigma_2}& \frac{1}{\Sigma_\varnothing^2}&0&0\\
			0&0&\frac{\Sigma_\varnothing}{\Sigma_1}&\frac{\Sigma_2}{\Sigma_{12}}\\
			\frac{\Sigma_\varnothing}{\Sigma_2}&0&0&\frac{\Sigma_1}{\Sigma_{12}}
		\end{bmatrix},
		\]
		whose Pl\"ucker coordinates are 
		\begin{equation} \label{eq:pluckn2}
			\Delta_{123}=\frac{1}{\Sigma_1 \Sigma_2},\Delta_{124}=\frac{1}{\Sigma_\varnothing \Sigma_{12}}, \Delta_{134}=\frac{1}{\Sigma_1 \Sigma_2}, \Delta_{234}=\frac{1}{\Sigma_\varnothing \Sigma_{12}}.
		\end{equation}
		Consider the electrical network in Figure~\ref{fig:facenetwork2}(a). Using Figure~\ref{fig:facenetwork2}(c), we compute
		\[
		\Meas_{G_+} \circ j_G^+ \circ q_G \circ \Psi_G(X,s_+,s_-) = \left[e_{123}+ \frac{\Sigma_1 \Sigma_2}{\Sigma_\varnothing \Sigma_{12}} e_{124}+ e_{134}+\frac{\Sigma_1 \Sigma_2}{\Sigma_\varnothing \Sigma_{12}} e_{234}\right],
		\]
		which agrees with (\ref{eq:pluckn2}) upon multiplying by $\Sigma_1 \Sigma_2$.
	\end{example}
	
	\section{The electrical left twist} \label{sec:elecleft}
	In this section, we define the electrical left twist and prove~\cref{thm:intro2}. By \cref{thm:ms1}, the right twist is a homeomorphism \[\vec \tau : \widetilde{\operatorname{Gr}}_{>0}(n+1,2n)/\Rpos \cong \Grpos \xrightarrow[]{\sim} \Grpos\] whose inverse is the left twist. We look for a similar statement for the electrical right twist. The dimension of $\OGpos$ is $\binom{n+1}{2}+1$ \cite[Lemma 5.7]{HenriquesSpeyer}, whereas the dimension of $\IG$ is $\binom{n}{2}$ (since this is the number of edges in $G$, hence the dimension of $\mathcal R_G$), so 
	\[
	\dim \OGpos - \dim \IGpos = n+1.
	\]
	We will see that there is an action of $\Rpos^{n+1}$ on $\OGpos$ preserving $\vec{\tau}_{{\rm elec}}$.
	
	We define an action of $\Rpos \times \Rpos^n$ on $\mathcal B_{G}$ as follows. For $s \in \Rpos$ and $t = (t_1,\dots,t_n) \in \Rpos^n$, 
	\begin{equation} \label{eq:Action}
		((s,t) \cdot B)_v := s  \left( \frac{ \prod_{i \in J(v)} t_i}{\prod_{i \notin J(v)} t_i}\right)B_v. 
	\end{equation}
Consider also the action of $\Rpos \times \Rpos^n$ on $\OGpos$ defined by:
	\begin{enumerate}
		\item $\Rpos$ acts on $\OGpos$ by $s \cdot (X,s_+,s_-) := (X,s s_+,s s_-)$.
		\item Recall from (\ref{maxtor}) the maximal torus $(\C^\times)^n$ inside $\Spin(Q)$ which has the parameterization $c: (\C^\times)^n \ra \Spin(Q)$. Restricting to $\Rpos^n \subset (\C^\times)^n$, we get a copy of $\Rpos^n$ inside $\Spin(Q)$ parameterized by $c: \Rpos^n \ra \Spin(Q)$. We have the action $t \cdot (X,s_+,s_-) = (X \rho(c(t))^T,c(t)s_+,c(t)s_- )$, where $\rho(c(t)) \in \SO(Q)$ is $\text{diag}(t_1^2,\dots,t_n^2,t_1^{-2},\dots,t_n^{-2})$.
	\end{enumerate}
{
\begin{lemma}\label{tauequivaction}
	The map $\Psi_G:\OGpos \ra \mathcal B_G$ is $\Rpos \times \Rpos^n$ equivariant.
\end{lemma}
\begin{proof}
	This follows from the observation that
	\[
	c_i(t_i) \cdot e_I = \begin{cases}
t_i e_I&~\text{if $i \in I$, and}\\
t_i^{-1} e_I&~\text{if $i \notin I$}.
	\end{cases}
	\]
\end{proof}
}
	
		\begin{figure}[h]
		\begin{tabular}{ccc}
			\includegraphics[width=0.3\textwidth]{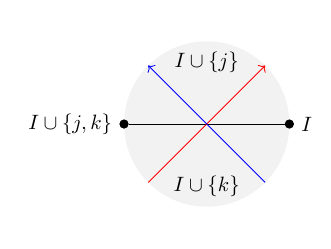} 
		\end{tabular}
		\caption{\label{fig:temperleyzz} {Labeling of the vertices and faces around an edge of $G$. The red and blue strands are $\beta_j$ and $\beta_k$ respectively. The strand $\beta_i, i \in \{j,k\}$, is oriented from $t_{n+i}$ to $t_i$. Locally any of the four choices of orientations of the two strands is possible; however, the other three cases are cyclic rotations of the one shown.
		}}
	\end{figure}

	\begin{lemma}\label{lem:qtequi}
		The map $q_G$ is invariant under the action (\ref{eq:Action}).
	\end{lemma}
	\begin{proof}
		Let $e=uv$ be an edge of $G$ with incident faces $f,g$. The map $q_G$ assigns to $e$ the conductance $\frac{B_u B_v}{B_f B_g}$. The four labels $(J(u), J(f),J(v),J(g))$ are some cyclic rotation of $(I, I \cup \{j\}, I \cup\{j,k\}, I \cup \{k\}),$ {(see Figure~\ref{fig:temperleyzz})} so the factors coming from the action of $(s,t)$ in the numerator and denominator cancel.
	\end{proof}
	
	By~\cref{thm:elecrt}, Lemma~\ref{tauequivaction} and~\cref{lem:qtequi}, $q_G$ and $\vec{\tau}_{{\rm elec}}$ descend to the quotients to yield the commuting diagram
	\[
	\begin{tikzcd} \mathcal B_G/\Rpos^{n+1} \arrow[r,"q_G"]    & \mathcal R_{G} \arrow[d,"\Meas_{G_+} \circ j_G^+","\sim"']\\ \OGpos/\Rpos^{n+1} \arrow[u,"\Psi_G","\sim"']
		\arrow[r,"\vec{\tau}_{{\rm elec}}"]
		& \IGpos
	\end{tikzcd},
	\]
	where each of the spaces has dimension $\binom{n}{2}$. We will show in~\cref{thm:elecrtltinv} that the two horizontal maps are also homeomorphisms.

	As in Section~\ref{sec:gtogplus}, let $\alpha_i$ denote the strand in $G_+$ from $d_{n+i-1}^-$ to $d_i^+$. Let $[\wt](\alpha_i)$ denote the alternating product of edge weights along $\alpha_i$, where the weights of edges oriented from black to white in $\alpha_i$ appear in the numerator and the weights of edges oriented from white to black in the denominator.
	
	\begin{lemma} \label{lem:wtaplha}
		If $[\wt] = j_G^+ \circ q_G(B)$ and $X = \Meas_{G_+}([\wt])$, then $[\wt](\alpha_i)=\frac{\Delta_{S(f_{n+i-1}^-)}(X)}{\Delta_{S(f_{n+i}^-)}(X)}=  \frac{B_{d_{i}}B_{d_{n+i}}}{B_{d_{i-1}}B_{d_{n+i-1}}}$.
	\end{lemma}
	\begin{proof}
		If $A:=\Phi_{G_+} \circ \cev\tau(X)$, then by~\cref{thm:MSmain}, $[\wt]=p_{G_+}(A)$. Let $d_{n+i-1}={\rm w}_1 \xrightarrow[]{e_1} {\rm b}_1 \xrightarrow[]{e_2} {\rm w}_2 \xrightarrow[]{e_3} {\rm b}_2 \xrightarrow[]{e_4} \cdots \xrightarrow[]{e_{2k-2}} {\rm w}_k \xrightarrow[]{e_{2k-1}} {\rm b}_k \xrightarrow[]{e_{2k}} {\rm w}_{k+1}=d_i$ denote the sequence of vertices and edges in $\alpha_i$. For each edge $e_i$, let $g_{i}^-$ (resp. $g_i^+$) denote the face of $G_+$ on the right (resp., left) of $e_i$. Notice that $g_{2j-1}^-=g_{2j}^-$ for $j \in [k]$ and $g_{2j}^+=g_{2j+1}^+$ for $j \in [k-1]$. Moreover, $g_1^+=f_{n+i}^-$ and $g_{2k}^-=f_i^-$. Therefore,
		\begin{align} \label{eq:awt}
			[\wt](\alpha_i) =\left(\frac{A_{g_1^+} A_{g_1^-}}{A_{f^-_{n+i-1}}}\right)\left(\frac{1}{A_{g_2^+} A_{g_2^-}}\right) \left(\frac{A_{g_3^+} A_{g_3^-}}{1}\right)\cdots \left(\frac{A_{f_i^-}}{A_{g_{2k}^+} A_{g_{2k}^-}}\right)=\frac{A_{f_{n+i}^-}}{A_{f_{n+i-1}^-}}.
		\end{align}
		Using~\cref{prop:msinv}(1), we get $[\wt](\alpha_i) = \frac{\Delta_{S(f_{n+i-1}^-)}(X)}{\Delta_{S(f_{n+i}^-)}(X)}$.

		Let $[\wt']:=p_{G_+} \circ i_G^+(B)$. Using (\ref{eq:awt}) for $[\wt']$, we get $[\wt'](\alpha_i)=\frac{i_G^+(B)_{f_{n+i}^-}}{i_G^+(B)_{f_{n+i-1}^-}}=\frac{B_{d_{n+i}}}{B_{d_{n+i-2}}}$. {By (\ref{eq::big}), we have} $[\wt]=t \cdot [\wt']$, where $t \in \Rpos^{2n}$ is given by $t_j = \frac{B_{d_{j-1}}}{B_{d_{j}}}$ for all $j \in [2n]$. Therefore, $\wt(e_1)=t_{n+i-1} \wt'(e_1)$ and $\wt(e_{2k})=t_{i} \wt'(e_{2k})$, so $[\wt](\alpha_i) = \frac{t_i}{t_{n+i-1}}[\wt'](\alpha_i)=\frac{B_{d_{i}}B_{d_{n+i}}}{B_{d_{i-1}}B_{d_{n+i-1}}}$.
		
	\end{proof}

	\begin{lemma} \label{lem:ltog}
		Given $X \in \IGpos$, let $t \in \Rpos^{2n}$ be such that $t_i t_{n+i} =\frac{\Delta_{S(f_{n+i}^-)}(X)}{\Delta_{S(f_{n+i-1}^-)}(X)}$. Then, $t \cdot \cev \tau(X) \in \operatorname{OG}_{>0}(n+1,2n)$.
	\end{lemma}
	
	\begin{proof}
		By~\cref{cor:elecsurj}, there exists $(Y,s_+,s_-) \in \OGpos$ such that $\vec\tau_{\rm elec}(Y,s_+,s_-) = X$. If $B:=\Psi_G(Y,s_+,s_-)$, then by~\cref{lem:wtaplha}, $t_i t_{n+i}=\frac{B_{d_{i-1}}B_{d_{n+i-1}}}{B_{d_{i}}B_{d_{n+i}}}$. Therefore, there exists $\lambda \in \Rpos^{2n}$ such that $t_i=\lambda_i \frac{B_{d_{i-1}}}{B_{d_i}}$ and $\lambda_{i+n} = \frac{1}{\lambda_i}$. Let $\mu \in \Rpos^{2n}$ be given by $\mu_i :=\frac{B_{d_{i-1}}}{B_{d_{i}}}$. By definition, $\vec \tau_{\rm elec} (Y,s_+,s_-) = \mu \cdot \vec \tau(Y)$, so by \cref{prop:msinv}(2) and \cref{thm:MSmain}, we have
		\[
		t \cdot \cev\tau(X) = t \cdot \cev \tau ( \mu \cdot \vec \tau(Y))=t \cdot \mu^{-1} \cdot \cev \tau(\vec \tau(Y)) = \lambda \cdot Y.
		\]
		Since $Y \in \operatorname{OG}_{>0}(n+1,2n)$ and $\lambda$ preserves $Q$, $\lambda \cdot Y \in \operatorname{OG}_{>0}(n+1,2n)$. 
	\end{proof}
	
	\begin{example} \label{example:lefttwist}
		Consider the electrical network $(G,c)$ in Figure \ref{fig:facenetwork2}(a). We compute \[\Meas_{G_+} \circ j_{G}^+(c) 
		= \left[e_{123}+ c e_{124}+ e_{134}+c e_{234}\right] \in \operatorname{IG}^\Omega_{>0}(3,4),
		\] 
		which is $\pl(X)$ for $X=\text{row span}\begin{bmatrix} 
			0&1&0&-1\\
			1&0&0&c\\
			0&0&-1&-c
		\end{bmatrix}.$ We have
		\[
		S(f_1^-)=123, S(f_2^-)=234, S(f_3^-)=134~\text{and}~S(f_4^-)=124,
		\]
		so we need to choose $t \in \Rpos^4$ such that \[
		t_1 t_3 = \frac{\Delta_{134}(X)}{\Delta_{234}(X)}=\frac 1 c~\text{and}~t_2 t_4= \frac{\Delta_{124}(X)}{\Delta_{134}(X)}=c,
		\]
		so $t_1=\frac{1}{c t_3}$ and $t_2=\frac{c}{t_4}$. Then, we compute
		\[
		t \cdot \cev \tau(X)=\text{row span} \begin{bmatrix}
			\frac{1}{t_3}& \frac{c}{t_4}&0&0\\
			\frac{1}{t_3 c} & 0&0& \frac{t_4}{c}\\
			0&-\frac{1}{t_4}& -t_3& 0
		\end{bmatrix}.
		\]
		To check that $t \cdot \cev \tau(X) \in \operatorname{OG}_{>0}(3,4)$, we compute the orthogonal complement $(t \cdot \cev \tau(X))^\perp = \Span(v)$, where $v= \left(\frac{1}{t_3t_4}, \frac{c}{t_4^2},\frac{c t_3}{t_4}, 1 \right)$, and check that $Q(v,v)=\frac{1}{t_3 t_4} \cdot \frac{c t_3}{t_4}-\frac{c}{t_4^2} \cdot 1=0$.
	\end{example}
	\begin{definition} \label{def:elt}
		Given $X \in \IGpos$, let $t \in \Rpos^{2n}$ be such that $t_i t_{n+i} =\frac{\Delta_{S(f_{n+i}^-)}(X)}{\Delta_{S(f_{n+i-1}^-)}(X)}$ and $t_{n+1}=1$, and let $Y:=t \cdot \cev \tau(X)$. By~\cref{lem:ltog}, $Y \in \operatorname{OG}_{>0}(n+1,2n)$. Let $(Y,s_+,s_-)$ be the lift of $Y$ to $\OGpos$ such that $\Sigma_\varnothing(Y,s_+,s_-)=\Sigma_{\{1\}}(Y,s_+,s_-)=1$. The \textit{electrical left twist} $\cev{\tau}_{\rm elec}: \IGpos \ra \OGpos/\Rpos^{n+1}$ is defined as $\cev\tau_{\rm elec}(X):=(Y,s_+,s_-)$. 
	\end{definition}

	\begin{theorem} \label{thm:elecrtltinv}
		The electrical left twist is well-defined in the sense that it is independent of the choice of $t \in \Rpos^{2n}$. The electrical right and left twists are mutually inverse homeomorphisms between $\OGpos/\Rpos^{n+1}$ and $\IGpos$ sitting in the commuting diagram
		\[
		\begin{tikzcd}[
			column sep={6em}
			] \mathcal B_G/\Rpos^{n+1} \arrow[r,"q_G","\sim"']    & \mathcal R_{G} \arrow[d,"\Meas_{G_+} \circ j_G^+","\sim"']\\ \OGpos/\Rpos^{n+1} \arrow[u,"\Psi_G","\sim"']
			\arrow[r,bend right=10,"\sim","{\vec{\tau}_{\rm elec}}"'] 
			& \IGpos \arrow[l,bend right=10,"\sim","{\cev{\tau}_{\rm elec}}"'] 
		\end{tikzcd},
		\]
		gluing which we get
		\[
		\begin{tikzcd}[
			column sep={6em}
			] \mathcal B_n/\Rpos^{n+1} \arrow[r,"q","\sim"']    & \mathcal R_{n} \arrow[d,"\Meas \circ j^+","\sim"']\\ \OGpos/\Rpos^{n+1} \arrow[u,"\Psi","\sim"']
			\arrow[r,bend right=10,"\sim","{\vec{\tau}_{\rm elec}}"'] 
			& \IGpos \arrow[l,bend right=10,"\sim","{\cev{\tau}_{\rm elec}}"'] 
		\end{tikzcd}.
		\]
	\end{theorem}
	\begin{proof}
	{	If $t'\in \Rpos^{2n}$ is another choice, define $\lambda \in \Rpos^{2n}$ by $\lambda_i := \frac{t_i'}{t_i}$ for all $i \in [2n]$. Note that $\lambda_1 =\lambda_{n+1}=1$.
		
		 Let $Y':= t' \cdot \cev \tau(X)=\lambda \cdot Y$ and let $(Y',s_+',s_-')$ denote its lift to $\OGpos$ such that $\Sigma_\varnothing(Y',s_+',s_-')=\Sigma_{\{1\}}(Y',s_+',s_-')=1$.
		By Lemma~\ref{tauequivaction},  \[(\sqrt{\lambda_2 \cdots \lambda_n},(1,\sqrt{\lambda_2},\dots,\sqrt{\lambda_n})) \cdot (Y,s_+,s_-) = (Y',s_+',s_-'),\]
		so $(Y,s_+,s_-)$ and $(Y',s_+',s_-')$ are in the same $\Rpos^{n+1}$ orbit. Therefore, $\cev\tau_{\rm elec}(X)$ is well-defined.
		
		Given $(Y,s_+,s_-) \in \OGpos$, we can use the action of $\Rpos^{n+1}$ to make $\Sigma_\varnothing(Y,s_+,s_-)=\Sigma_{\{1\}}(Y,s_+,s_-)=1$; indeed, we act by
		\[
	\left(\frac{1}{\sqrt{\Sigma_\varnothing(Y,s_+,s_-) \Sigma_{\{1\}}(Y,s_+,s_-)}}, \left(\sqrt{\frac{\Sigma_\varnothing(Y,s_+,s_-)}{\Sigma_{\{1\}}(Y,s_+,s_-)}},1,\dots,1\right)\right) \in \Rpos^{n+1}.	
		\]
				 If we choose $t_i := \frac{\Sigma_{J(d_{i-1})}(Y,s_+,s_-)}{\Sigma_{J(d_i)}(Y,s_+,s_-)}$ to define the electrical left twist, then $\cev\tau_{\rm elec} \circ \vec\tau_{\rm elec}(Y,s_+,s_-) = (Y,s_+,s_-)$, so $\vec\tau_{\rm elec}$ is injective with left inverse $\cev \tau_{\rm elec}$. By~\cref{cor:elecsurj},   $\vec\tau_{\rm elec}$ is also surjective, so $\cev \tau_{\rm elec}$ is the two-sided inverse.
	}
	\end{proof}
	
	\begin{example}
		Recall~\cref{example:lefttwist} and set $t_3=1$. Using row operations, we can write \[
		Y:=t \cdot \cev \tau(X)=\text{row span}\begin{bmatrix}
			1 & \frac{ c}{t_4}&0&0\\
			0& \frac{1}{ t_4} &1&0\\
			0&-\frac{c}{t_4^2}&0&1
		\end{bmatrix}.\]
		Letting $\Sigma_\varnothing(Y,s_+,s_-) =\Sigma_1(Y,s_+,s_-)=1$ and comparing with the matrix in~\cref{example:rt2}, we see that $\Sigma_{2}(Y,s_+,s_-)=\frac{ c}{t_4}$ and $\Sigma_{12}(Y,s_+,s_-)=\frac{1}{ t_4}$. Therefore, $q_G \circ \Psi_G (Y,s_+,s_-)$ assigns to the edge the conductance
		\[
		\frac{\Sigma_1(Y,s_+,s_-) \Sigma_2(Y,s_+,s_-)}{\Sigma_\varnothing(Y,s_+,s_-) \Sigma_{12}(Y,s_+,s_-)}=\frac{1 \cdot \frac{c}{t_4}} {1\cdot \frac{1}{t_4}}=c,
		\]
		verifying commutativity of the diagram in~\cref{thm:elecrtltinv}.
	\end{example}
	
	\section{An example of the inverse map} \label{sec:recon}
	
	In this section, we work out in detail the inverse map when $n=3$. For background on electrical networks, the Laplacian and the response matrix, see \cite{kensurvey}. Let $(G,c)$ denote the electrical network in Figure~\ref{fig:facenetwork3}(a). The Laplacian is
	\[
	\Delta=\begin{blockarray}{ccccc}
		b_1 & b_2 & b_3 & u \\
		\begin{block}{[cccc]c}
			a & 0 & 0 & -a &b_1 \\
			0 & b & 0 & -b &b_2 \\
			0 & 0 & c & -c &b_3\\
			-a & -b & -c & a+b+c &u\\
		\end{block}
	\end{blockarray},
	\]	
	from which the response matrix is obtained as the Schur complement
	\begin{equation}\label{eq:respnet3}
		L=-\begin{bmatrix}
			a & 0 & 0  \\
			0 & b & 0   \\
			0 & 0 & c  
		\end{bmatrix}+\begin{bmatrix}
			-a\\-b\\-c
		\end{bmatrix} \begin{bmatrix}
			a+b+c
		\end{bmatrix}^{-1} \begin{bmatrix}
			-a & -b & -c
		\end{bmatrix}=\begin{bmatrix}
			-\frac{a (b+c)}{a+b+c} & \frac{a b}{a+b+c} & \frac{a c}{a+b+c} \\
			\frac{a b}{a+b+c} & -\frac{b (a+c)}{a+b+c} & \frac{b c}{a+b+c} \\
			\frac{a c}{a+b+c} & \frac{b c}{a+b+c} & -\frac{c (a+b)}{a+b+c} \\
		\end{bmatrix}.
	\end{equation}
	
	By~\cite[Theorem 1.8]{CGS}, the point $X:=\pl^{-1} \circ \Meas_{G_+} \circ j_G^+(c) \in \operatorname{IG}^\Omega_{>0}(4,6)$ is 
	\[ 
	\text{row span}\begin{bmatrix}
		0&1&0&-1&0&1 \\
		1&0&0&L_{12}&0&-L_{12}-L_{13} \\
		0&0&-1&-L_{12}-L_{23}&0&L_{12} \\
		0&0&0&L_{23}&1&L_{13} \\
	\end{bmatrix}. 
	\]
	Using the face labels that have been computed in Figure~\ref{fig:network3labels}(b), to define the electrical left twist, we need to choose $t \in \Rpos^6$ such that
	\[
	t_1 t_4 = \frac{\Delta_{1456}(X)}{\Delta_{3456}(X)}=\frac{L_{23}}{L_{13}}, t_2 t_5 = \frac{\Delta_{1256}(X)}{\Delta_{1456}(X)}=\frac{L_{12}}{L_{23}}, t_3 t_6 = \frac{\Delta_{1236}(X)}{\Delta_{1256}(X)}=\frac{L_{13}}{L_{12}}~\text{and}~t_4=1, 
	\]
	so let us take $t_1 = \frac{L_{23}}{L_{13}}, t_2 = {L_{12}}, t_3={L_{13}}, t_4=1, t_5=\frac{1}{ L_{23}}$ and $t_6=\frac{1}{L_{12}}$. We compute
	\[
	Y:=	t \cdot \cev\tau(X)=\text{row span}\begin{bmatrix}\frac{L_{12} L_{13}+L_{12} L_{23}+L_{13} L_{23}}{L_{13}} & L_{12} & 0 & 0 & 0 & -\frac{1}{L_{13}} \\
		\frac{L_{23}}{L_{13}} & 0 & 0 & 0 & -\frac{1}{L_{12}} & -\frac{1}{L_{12} L_{13}} \\
		-1 & -1 & -L_{13} & 0 & 0 & 0 \\
		0 & 0 & L_{12} & \frac{1}{L_{23}} & \frac{1}{L_{23}} & 0 \\
	\end{bmatrix}.
	\]
	The skew symmetric matrices $M_+$ and $M_-$ as in~\cref{sec:pfformulas} are
	\begin{align*}
		M_+ &= \begin{bmatrix}
			0 & L_{23} & L_{12} L_{13}+L_{12} L_{23}+L_{13} L_{23} \\
			-L_{23} & 0 & L_{12} L_{13} \\
			-(L_{12} L_{13}+L_{12} L_{23}+L_{13} L_{23}) & -L_{12} L_{13} & 0 \\
		\end{bmatrix}~\text{and}\\M_-&=\begin{bmatrix}
			0 & 1 & L_{13} \\
			-1 & 0 & -L_{12} L_{23} \\
			-L_{13} & L_{12} L_{23} & 0 \\
		\end{bmatrix}.
	\end{align*} 	
	
	\begin{figure}
		
		\includegraphics[width=0.6\textwidth]{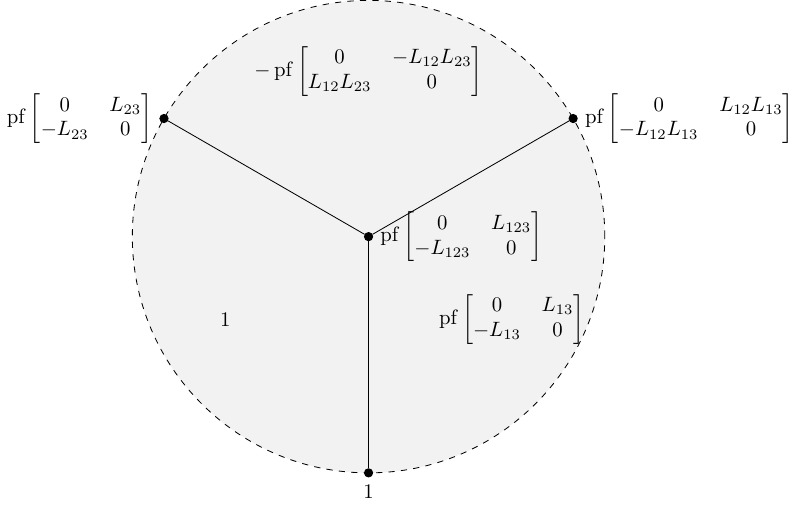}
		
		\caption{\label{fig:pfnet3} $\Psi_G\circ \cev \tau_{\rm elec}(X)$, where $L_{123}:= L_{12} L_{13}+L_{12} L_{23}+L_{13} L_{23}$.}
	\end{figure}

	Using the labels in Figure~\ref{fig:network3labels}(a) and~\cref{prop:pfaffian}, we get that $\Psi_G\circ \cev \tau_{\rm elec}(X)$ is as shown in Figure~\ref{fig:pfnet3}, so $q_G \circ \Psi_G\circ \cev \tau_{\rm elec}(X)$ is given by
	\begin{align*}
		c(u b_1)&=-\frac{\pf \begin{bmatrix}
				0&L_{12} L_{13}+L_{12} L_{23}+L_{13} L_{23}\\
				-(L_{12} L_{13}+L_{12} L_{23}+L_{13} L_{23})&0	
			\end{bmatrix}	\pf \begin{bmatrix}
				0&L_{12} L_{13}\\
				-L_{12} L_{13}&0	
			\end{bmatrix}
		}{\pf \begin{bmatrix}
				0&L_{13}\\
				-L_{13}&0
			\end{bmatrix} \pf \begin{bmatrix}
				0&-L_{12}L_{23}\\
				L_{12}L_{23}&0
		\end{bmatrix}}\\
		&= \frac{L_{12} L_{13}+L_{12} L_{23}+L_{13} L_{23}}{L_{23}},\\
		c({ub_2})&=\frac{\pf \begin{bmatrix}
				0&L_{12} L_{13}+L_{12} L_{23}+L_{13} L_{23}\\
				-(L_{12} L_{13}+L_{12} L_{23}+L_{13} L_{23})&0	
		\end{bmatrix}}{\pf \begin{bmatrix}
				0&L_{13}\\-L_{13}&0
		\end{bmatrix}}\\
		&=\frac{L_{12} L_{13}+L_{12} L_{23}+L_{13} L_{23}}{L_{13}},\\
		c(ub_3)&=-\frac{\pf \begin{bmatrix}
				0&L_{12} L_{13}+L_{12} L_{23}+L_{13} L_{23}\\
				-(L_{12} L_{13}+L_{12} L_{23}+L_{13} L_{23})&0	
			\end{bmatrix}\pf \begin{bmatrix}
				0&L_{23}\\-L_{23}&0
		\end{bmatrix}}{\pf \begin{bmatrix}
				0&-L_{12}L_{23}\\L_{12}L_{23}&0
		\end{bmatrix}}\\
		&=\frac{L_{12} L_{13}+L_{12} L_{23}+L_{13} L_{23}}{L_{12}}.
	\end{align*}	 
	From (\ref{eq:respnet3}), we have
	$
	L_{12}=\frac{ab}{a+b+c}, L_{13}=\frac{ac}{a+b+c},L_{23}=\frac{bc}{a+b+c}
	$, so $L_{12} L_{13}+L_{12} L_{23}+L_{13} L_{23}= \frac{abc}{a+b+c}$. Plugging in these formulas, we get $c(ub_1)=a, c(ub_2)=b,c(ub_3)=c$.
	
\appendix

{
	\section{Notation}
	In the appendix, we collect some of the notation for the spaces and maps used in the paper and the main commutative diagrams in which they sit. The third column of an entry indicates where it first appears.
\subsection{Grassmannians}
	\begin{center}
	\begin{tabular}{  m{0.3\textwidth}  m{0.6\textwidth} m{0.1\textwidth}  } 
$\Grkn$ & Grassmannian  & Section~~\ref{sec:grassmann}\\
$\Grkndec$ & decorated Grassmannian  & Section~~\ref{sec:grassmann}\\
$\Grknpos$ & positive Grassmannian&Section~~\ref{sec:grassmann}\\
$\Grkndecpos$ & positive decorated Grassmannian&Section~~\ref{sec:grassmann}\\
$\IGpos$ & positive Lagrangian Grassmannian &  Section~\ref{sec:gtogplus}\\
$\operatorname{OG}_{>0}(n+1,2n)$ & positive orthogonal Grassmannian&
Definition~\ref{def:og}\\
$\OGpos$ & positive decorated orthogonal Grassmannian&
Definition~\ref{def:og}\\
$\Delta_I$ & Pl\"ucker coordinate & Section~~\ref{sec:grassmann}\\
$\Sigma_J$ & Cartan coordinate & Section~\ref{sec:og}

	\end{tabular}
\end{center}
	
\subsection{Bipartite graphs}

	\begin{center}
		\begin{tabular}{  m{0.3\textwidth}  m{0.6\textwidth} m{0.1\textwidth}  } 
	$\mathcal X_\Gamma$ & space of edge weights modulo gauge on $\Gamma$ & Section~\ref{sec:dimerboundary}\\
	$\Meas_\Gamma:\mathcal X_\Gamma \rightarrow \Grknpos$ & boundary measurement map & Section~\ref{sec:dimerboundary}\\
	$\mathcal A_\Gamma$ & space parameterized by $A$ variables & Section~\ref{sec:Avariables}\\
	$\Phi_\Gamma:\Grkndecpos \ra \mathcal A_\Gamma$ & Scott's map & Section~\ref{sec:Avariables}\\
	$p_\Gamma:\mathcal A_\Gamma \ra \mathcal X_\Gamma$ & canonical map of cluster varieties & Definition~\ref{def:atox}\\
	$\vec \tau$ and $\cev \tau$ & right and left twists & Definition~\ref{def:twist}\\
	& $ 
		\begin{tikzcd} \mathcal A_\Gamma/\Rpos   \arrow[r,"p_\Gamma","\sim"'] & \mathcal X_\Gamma \arrow[d,"\sim"',"\Meas_\Gamma"]\\ \Grknpos \arrow[u,"\sim"',"\Phi_\Gamma"]\arrow[r,bend right=10,"{\vec{\tau}}"',"\sim"]&\Grknpos\arrow[l,bend right=10,"\sim","{\cev{\tau}}"'] 
		\end{tikzcd} 
	$ & Theorem~\ref{thm:MSmain}
		\end{tabular}
	\end{center}
}

{
\subsection{Electrical networks}

	\begin{center}
	\begin{tabular}{  m{0.3\textwidth}  m{0.6\textwidth} m{0.1\textwidth}  } 
		$\mathcal R_G$ & space of conductances on $G$ & Section~\ref{sec:gtogplus}\\
		$G_+$ & weighted bipartite graph associated to $G$& Section~\ref{sec:gtogplus}\\
		$j_G^+: \mathcal R_G \hookrightarrow \mathcal X_{G_+}$ & generalized Temperley's bijection  & Section~\ref{sec:gtogplus}\\
		$\mathcal B_G$ & space parameterized by $B$ variables & Section~\ref{sec:bvar}\\
		$i_G^+:\mathcal B_{G} \hookrightarrow \mathcal A_{G_+}$ & $A$ variable $=$ product of two $B$ variables &Definition~\ref{def:atob}\\
		$\Psi_G: \OGdec \ra \mathcal B_G$ &Henriques and Speyer's map &Section~\ref{sec:bvar}\\
		$q_G: \mathcal B_G \ra \mathcal R_G$ & canonical map from $B$ variables and conductances & Definition~\ref{def:q}\\
		$\vec{\tau}_{\rm elec}$ & electrical right twist &Definition~\ref{def:elecrt}\\
	${\cev{\tau}_{\rm elec}}$ &electrical left twist &Definition~\ref{def:elt}\\
	& $\begin{tikzcd} \mathcal R_G  \arrow[r,hook,"j_{G}^+"]\arrow[d,"\sim","\Meas_{G_+} \circ j_G^+"'] & \mathcal X_{G_+} \arrow[d,"\sim"',"\Meas_{G_+}"]\\ \IGpos \arrow[r,hook] &\Grpos
	\end{tikzcd} 	$& Section~\ref{sec:gtogplus} \\
& $	\begin{tikzcd} 
	\mathcal B_{G} \arrow[r,"i_{G}^+"] &\mathcal A_{G_+}\\
	\OGpos \arrow[r]\arrow[u,"\Psi_G","\sim"'] & \widetilde{\operatorname{Gr}}_{>0}(n+1,2n) \arrow[u,"\Phi_{G_+}"',"\sim"]
\end{tikzcd}$ & Proposition~\ref{prop:ogg}\\
& $\begin{tikzcd} \mathcal B_G   \arrow[r,"q_G"] & \mathcal R_{G} \arrow[d,"\sim"',"\Meas_{G_+} \circ j_{G}^+"]\\ \OGpos \arrow[u,"\Psi_G","\sim"']
	\arrow[r,"\vec\tau_{{\rm elec}}"]
	&\IGpos
\end{tikzcd} $ &  Section~\ref{sec:elecleft}\\
& $\begin{tikzcd}[
	column sep={6em}
	] \mathcal B_G/\Rpos^{n+1} \arrow[r,"q_G","\sim"']    & \mathcal R_{G} \arrow[d,"\Meas_{G_+} \circ j_G^+","\sim"']\\ \OGpos/\Rpos^{n+1} \arrow[u,"\Psi_G","\sim"']
	\arrow[r,bend right=10,"\sim","{\vec{\tau}_{\rm elec}}"'] 
	& \IGpos \arrow[l,bend right=10,"\sim","{\cev{\tau}_{\rm elec}}"'] 
\end{tikzcd}$ & Theorem~\ref{thm:elecrtltinv}
	\end{tabular}
\end{center}
}

	\bibliographystyle{alpha}
	\bibliography{biblio}
\end{document}